\documentclass[12pt,twoside,english,a4paper]{article}
\usepackage{babel,amssymb,amsthm}
\usepackage{graphicx}
\usepackage{amsmath}
\usepackage{bm}

\newcommand{\smallfrac}[2]{{\textstyle\frac{#1}{#2}}}

\newtheorem{proposition}{Proposition}[section]
\newtheorem{corollary}[proposition]{Corollary}
\newtheorem{lemma}[proposition]{Lemma}
\newtheorem{theorem}[proposition]{Theorem}

\numberwithin{equation}{section}

\title{Coupling of HDG with a double-layer potential BEM}

\date{\today}

\author{Zhixing Fu\footnote{Department of Mathematical Sciences, University of Delaware, Newark DE 19716. {\tt zfu@math.udel.edu}}, Norbert Heuer\footnote{Facultad de Matem\'aticas, Ponticia Universidad Cat\'olica de Chile, Avenida Vicu\~na Mackenna 4860,
Santiago, Chile {\tt nheuer@mat.puc.cl}. Partially supported by FONDECYT project 1110324 and CONICYT Anillo ACT1118 (ANANUM).}, \& Francisco--Javier Sayas\footnote{Department of Mathematical Sciences, University of Delaware, Newark DE 19716.  {\tt fjsayas@udel.edu}. Partially funded by NSF (grant DMS 1216356)}}

\begin{document}

\maketitle

\begin{abstract}
In this paper we propose and analyze a new coupling procedure for the Hybridizable Discontinuous Galerkin Method with Galerkin Boundary Element Methods based on a double layer potential representation of the exterior component of the solution of a transmission problem. We show a discrete uniform coercivity estimate for the non-symmetric bilinear form and prove optimal convergence estimates for all the variables, as well as superconvergence for some of the discrete fields. Some numerical experiments support the theoretical findings.
\end{abstract}

\section{Introduction}

In this paper we propose and analyze a new coupling procedure for the Hybridizable Discontinuous Galerkin Method (HDG) \cite{CoGoLa:2009} and a Galerkin Boundary Element procedure based on a double layer potential representation. The model problem is a transmission problem in free space, coupling a linear diffusion equation with variable diffusivity ($\mathrm{div}\,\kappa \nabla u=f$)  in a polygonal domain, with an exterior Laplace equation. The transmission conditions are given by imposing the value of the difference of the unknown and its flux on the interface between the interior and the exterior domains. 

Much has been written on the advantages and disadvantages of the many Discontinuous Galerkin schemes. In support of this piece of work, let us briefly hint at some of the features that make the HDG --originally a derivation of the Locally Discontinuous Galerkin (LDG) method by Bernardo Cockburn and his collaborators-- a family of interest. First of all, HDG uses polynomials of the same degree $k\ge 0$ for discretization of all the variables --in the case of diffusion problems, a scalar unknown $u$, the vector valued flux $\mathbf q=-\kappa\nabla u$ and a scalar unkown on the skeleton of the triangulation--, attaining {\em optimal order} $\mathcal O(h^{k+1})$ in the approximation of all of them. This feature compares well with Mixed Finite Element Methods, of which HDG is a natural modification: HDG can be considered as a variant of the Raviart-Thomas and Brezzi-Douglas-Marini Mixed Elements, implemented with Lagrange multipliers on the interelement faces \cite{ArBr:1985}, eliminating the degrees of freedom used for stabilization and using a stabilization (not penalization) parameter to perform the same task. Similar to the Lagrange multiplier implementation of mixed methods by Arnold and Brezzi \cite{ArBr:1985}, HDG is implemented by {\em hybridization}, reducing the unknowns to those living on interelement faces. This reduces considerably the number of degrees of freedom and makes the method competitive with respect to other families of Finite Element Methods, while still being a mixed method, that approximates several fields simultaneously. Another interesting feature of HDG is the fact that for polynomial degrees $k\ge 1$, the two unknowns related to the scalar field (the respective approximations inside the element and on the skeleton) {\em superconverge} at rate $\mathcal O(h^{k+2})$. This allows for the application of standard postprocessing techniques \cite{CoGoSa:2010} that can be traced back to Stenberg \cite{Stenberg:1991}.

Among many of its good properties (for the limited set of equations where it is usable), the Boundary Element Method provides a reliable form of constructing high order absorbing boundary conditions with complete flexibility on the geometric structure of the domain (it does not require the domain to be convex or even connected). It is therefore natural to test and study the possibility of using BEM as a way of generating a coupled discretization scheme with HDG for transmission problems. The coupling of DG and BEM started with the study of LDG-BEM schemes \cite{GaSa:2006}, extended to other DG methods of the Interior Penalty (IP) family \cite{GaHeSa:2010}. All of these methods, plus several new ones, were presented as particular cases of a methodology for creating {\em symmetric couplings} of DG and BEM in \cite{CoSa:2012}. One of the methods in this latter paper, using HDG and BEM, was recently analyzed in \cite{CoGuSa:TA}. As opposed to symmetric couplings with BEM, that need two integral equations (and thus four integral operators) to reach a stable formulation, non-symmetric couplings require only one integral equation (and two integral operators). From this point of view, they provide simpler (and more natural) forms of coupling BEM with field methods (FEM, Mixed FEM or DG). The non-symmetric coupling of Mixed FEM and BEM is recent \cite{MeSaSe:2011} and expands ideas used for analyzing the simple coupling of FEM-BEM \cite{Sayas:2009}. The presentation and testing of non-symmetric coupling of DG methods of the IP family with BEM \cite{OfRoStTa:2012}, led to its analysis \cite{HeSa:SB}, providing the first rigorous proof of convergence of this kind of schemes. Note that, unlike in the case of symmetric couplings, where the stability analysis boils down to an energy argument, coercivity properties in non-symemtric couplings pose a serious analytic challenge. The present paper makes a contribution in that direction, showing how to develop a coercivity analysis for non-symmetric coupling of HDG-BEM, and pointing out some interesting facts about the size of the diffusion parameter in the interior domain.

The paper is structured as follows. In Section \ref{sec:1} we present the method. Sections \ref{sec:2} to \ref{sec:4} cover the analysis in progressive steps. Section \ref{sec:F} discusses several related extensions and possible modifications of the method. In Section \ref{sec:5} we include some numerical experiments. A collection of known  results concerning transmission problems, potentials and integral operators, is given in Appendix \ref{sec:A} for ease of reference.

\paragraph{Regarding notation.} Basic theory of Sobolev spaces  is assumed throughout. Norms of $L^2(\mathcal O)$ type will be subscripted with the integration domain $\|\cdot\|_{\mathcal O}$. All other norms will be indexed with the name of the space.

\section{Formulation and discretization}\label{sec:1}

Let $\Omega\subset \mathbb R^d$ be a polygonal domain when $d=2$ or a Lipschitz polyhedral domain when $d=3$, with boundary $\Gamma$. For simplicity we will assume that $\Gamma$ is connected. (Connectedness plays a minor role in the amount of energy-free solutions of the transmission problem.) Let $\Omega_+:=\mathbb R^d\setminus\overline\Omega$ be then the domain exterior to $\Gamma$. The outward pointing unit normal vector field on $\Gamma$ is denoted $\mathbf n$. 

\subsection{Statement of the problem}

We consider a scalar positive diffusion coefficient $\kappa \in L^\infty(\Omega)$ such that $\kappa^{-1}\in L^\infty(\Omega)$.
\begin{subequations}\label{eq:1.1}
We are interested in the following transmission problem: an elliptic second order diffusion equation in $\Omega$, written as a first order system,
\begin{equation}\label{eq:1.1a}
\mathbf q+\kappa \nabla u =\mathbf 0\qquad \mbox{ and } \qquad 
\mathrm{div}\,\mathbf q = f \qquad \mbox{in $\Omega$},
\end{equation}
the Laplace equation in the exterior domain
\begin{equation}\label{eq:1.1b}
\Delta u_+=0 \quad\mbox{in $\Omega_+$} \qquad \mbox{and} \qquad u=o(1) \quad\mbox{at infinity},
\end{equation}
coupled through two transmission conditions on the common boundary of these domains
\begin{equation}\label{eq:1.1c}
u=u_++\beta_0 \qquad\mbox{and}\qquad -\mathbf q\cdot\mathbf n = \partial_{\mathbf n} u_++\beta_1 \qquad \mbox{ on $\Gamma$.}
\end{equation}
\end{subequations}
In principle,  boundary values of $u$ and $u_+$ in \eqref{eq:1.1c} are taken in the sense of traces for functions with local Sobolev $H^1$ regularity. Similarly $\mathbf q\cdot\mathbf n$ and $\partial_{\mathbf n} u_+$ are defined in a weak form as elements of the space $H^{-1/2}(\Gamma)$. Basic regularity requirements for data is:
$f\in L^2(\Omega)$, $\beta_0\in H^{1/2}(\Gamma)$, and  $\beta_1\in L^2(\Gamma).$
We note that $\beta_1 \in H^{-1/2}(\Gamma)$ is valid data for the transmission problem, but not for the type of formulation we are going to use. Data functions are assumed to satisfy the following compatibility condition:
\begin{equation}\label{eq:1.2}
\int_\Omega f+\int_\Gamma \beta_1 =0.
\end{equation}
A more detailed discussion about this condition and solvability issues for the transmission problems is given in Section \ref{sec:A.3}. Here we just note the following:
\begin{itemize}
\item[(a)] When $d=2$, condition \eqref{eq:1.2} is necessary and sufficient for the transmission problem \eqref{eq:1.1} to have a solution. In this case $u_+=\mathcal O(r^{-1})$ as $r=|\mathbf x|\to \infty$.
\item[(b)] When $d=3$, problem \eqref{eq:1.1} has a unique solution even if \eqref{eq:1.2} fails to hold. If \eqref{eq:1.2} is satisfied, the solution decays as $u_+=\mathcal O(r^{-2})$ at infinity. Proposition \ref{prop:A.4}(b) shows the elementary modification of data that is needed to have condition \eqref{eq:1.2} satisfied, while still providing a solution of the original problem.
\end{itemize}

\subsection{Coupling with a double layer potential}

Condition \eqref{eq:1.2} allows us to write $u_+$ as a double layer potential
\begin{equation}
u_+=\mathcal D \varphi, \qquad \varphi \in H^{1/2}_0(\Gamma).
\end{equation}
(See \eqref{eq:A.00} and \eqref{eq:A.0} for the definitions and Proposition \ref{prop:A.22} for the representation result.) The exterior trace and normal derivative of the double layer potential can be written in terms of boundary integral operators, as explained in Section \ref{sec:A.4}. These operators are introduced in \eqref{eq:A.11}-\eqref{eq:A.12}. They allow us to write the transmission conditions \eqref{eq:1.1c} in the equivalent form
\begin{equation}
u=\smallfrac12\varphi+\mathcal K \varphi+\beta_0 \qquad \mbox{and}\qquad \mathbf q\cdot\mathbf n=\mathcal W \varphi -\beta_1 \qquad \mbox{on $\Gamma$}.
\end{equation}
Note that
\begin{equation}\label{eq:1.0}
\int_\Gamma \varphi = 0 \qquad \mbox{and}\qquad \mathcal W \varphi= \mathbf q\cdot\mathbf n+\beta_1 \qquad \mbox{on $\Gamma$}
\end{equation}
imply
\begin{equation}
\omega(\varphi,\phi):=\langle\mathcal W \varphi,\phi\rangle_\Gamma+\int_\Gamma \varphi\,\int_\Gamma \phi = \langle \mathbf q\cdot\mathbf n+\beta_1,\phi\rangle_\Gamma \qquad \forall \phi \in H^{1/2}(\Gamma).
\end{equation}
At this moment, it is convenient to collect all the equations that make up the formulation we are going to discretize:
\begin{subequations}\label{eq:1.6}
\begin{alignat}{4}
\label{eq:1.6a}
\kappa^{-1}\mathbf q +\nabla u & = 0 & & \mbox{ in $\Omega$},\\
\label{eq:1.6b}
\mathrm{div}\,\mathbf q &=f & & \mbox{ in $\Omega$},\\
\label{eq:1.6c}
u - (\smallfrac12\varphi+\mathcal K\varphi) &=\beta_0 & & \mbox{ on $\Gamma$},\\
\label{eq:1.6d}
-\langle\mathbf q\cdot\mathbf n,\phi\rangle_\Gamma + \omega(\varphi,\phi)&=\langle\beta_1,\phi\rangle_\Gamma &\qquad & \forall \phi\in H^{1/2}(\Gamma).
\end{alignat}
\end{subequations}
The exterior field is reconstructed as a double layer potential $u_+=\mathcal D \varphi$. The next lemma ensures that the integral boundary condition   and the normalization condition \eqref{eq:1.0} for the density $\varphi$  are adequately encoded in the system \eqref{eq:1.6}.

\begin{lemma}\label{lemma:1.1}
If $(\mathbf q,u,\varphi)$ is a solution to \eqref{eq:1.6} and the compatibility condition \eqref{eq:1.2} is satisfied, then $\int_\Gamma \varphi=0$.
\end{lemma}

\begin{proof}
Integrating \eqref{eq:1.6b} over $\Omega$ and testing \eqref{eq:1.6d} with $\phi=1$, it follows that
\[
\int_\Omega f = \langle \mathbf q\cdot\mathbf n,1\rangle_\Gamma=\omega(\varphi,1)-\langle\beta_1,1\rangle_\Gamma=|\Gamma| \int_\Gamma \varphi-\int_\Gamma \beta_1,
\]
because of Proposition \ref{prop:A.6}. This proves the result.
\end{proof}

\subsection{Discretization with HDG and Galerkin BEM}

We start the discretization process by describing the discrete geometric elements.  The domain $\Omega$ is divided into triangles ($d=2$) or tetrahedra ($d=3$) with the usual conditions for conforming finite element meshes. A general element will be denoted $K$, $\mathcal T_h$ will be the set of all elements and $h:=\max_{K\in \mathcal T_h} h_K$ the maximum diameter of elements of the triangulation. The triangulation is assumed to be shape-regular. Extension of the forthcoming results to non-conforming meshes is relatively simple while not straightforward, requiring the use of some finely tuned results about HDG on general grids \cite{ChCo:TAa}, \cite{ChCo:TAb}. 

We assume the existence of a triangulation $\Gamma_h$ of the boundary $\Gamma$, formed by line segments ($d=2$) or triangles ($d=3$). We write $h_\Gamma:=\max_{e\in \Gamma_h} h_e$. From the point of view of analysis, it is immaterial whether this grid is related to $\mathcal T_h$ or not. (Note that for the experiments we will only use meshes where $\Gamma_h$ is the trace of $\mathcal T_h$ on $\Gamma$, which makes the implementation considerably simpler.)
At the time of stating and proving the convergence theorems we will add some technical restrictions relating the mesh-sizes of the two grids.

The set of edges ($d=2$) or faces ($d=3$) of all elements of the triangulation will be denoted $\mathcal E_h$ and $\partial \mathcal T_h=\cup_{e\in \mathcal E_h} e$ will denote the skeleton of the triangulation. Volumetric integration will be denoted with parentheses
\[
(u,v)_K:=\int_K u\,v, \qquad (\mathbf p,\mathbf q)_K:=\int_K \mathbf p\cdot\mathbf q, \qquad (u,v)_{\mathcal T_h}:=\sum_{K\in \mathcal T_h} (u,v)_K,
\]
while integration on boundaries will be denoted with angled brackets
\[
\langle u,v\rangle_{\partial K}:=\int_{\partial K} u\, v, \qquad \langle u,v\rangle_{\partial \mathcal T_h}:=\sum_{K\in \mathcal T_h} \langle u,v\rangle_{\partial K}, \qquad \langle u,v\rangle_{\partial \mathcal T_h\setminus\Gamma} :=\sum_{K\in \mathcal T_h} \langle u,v\rangle_{\partial K\setminus\Gamma}.
\]
The unit normal outward pointing vector field on $\partial K$ will be denoted $\mathbf n_{\partial K}$, or simply $\mathbf n$, when there is no doubt on what the element is.

Local discrete spaces will be composed of polynomials. The set $\mathcal P_k(K)$ contains all $d$-variate polynomials of degree less than or equal to $k$ and $\boldsymbol{\mathcal P}_k(K):=(\mathcal P_k(K))^d$. If $e\in \mathcal E_h$ or $e\in \Gamma_h$, $\mathcal P_k(e)$ is the set of polynomials of degree less than or equal to $k$ defined on $e$, i.e.,  the space of $(d-1)$-variate polynomials on local tangential coordinates. Four global spaces will be used for discretization:
\begin{subequations}
\begin{alignat}{4}
\boldsymbol V_h :=& \{ \mathbf v:\Omega \to \mathbb R^d\,:\, \mathbf v|_K \in \boldsymbol{\mathcal P}_k(K) \quad \forall K \in \mathcal T_h\},\\
W_h :=& \{ w:\Omega \to \mathbb R\,:\, w|_K \in \mathcal P_k(K) \quad \forall K \in \mathcal T_h\},\\
M_h :=& \{ \widehat v:\partial \mathcal T_h\to \mathbb R\,:\, \widehat v|_e\in \mathcal P_k(e) \quad \forall e \in \mathcal E_h\},\\
Y_h :=& \{ \phi:\Gamma \to \mathbb R\,:\, \phi\in \mathcal C(\Gamma), \quad \phi|_e\in \mathcal P_{k+1}(e) \quad\forall e \in \Gamma_h\}.
\end{alignat}
\end{subequations}
The discretization method consists of using the Hybridizable Discontinuous Galerkin method (see \cite{CoGoLa:2009}, \cite{CoGoSa:2010}) for equations \eqref{eq:1.6a}-\eqref{eq:1.6b}-\eqref{eq:1.6c} (this part of the system can be understood as an interior Dirichlet problem) and conforming Galerkin (Boundary Element) Method for equation \eqref{eq:1.6d} (this part is considered as a hypersingular integral equation for an exterior Neumann problem). We thus look for $\mathbf q_h \in \boldsymbol V_h$, $u_h \in W_h$,  $\widehat u_h \in M_h$, and $\varphi_h \in Y_h$ satisfying
\begin{subequations}\label{eq:1.8}
\begin{alignat}{4} 
\label{eq:1.8a}
(\kappa^{-1}\mathbf q_h,\mathbf r)_{\mathcal T_h} -(u_h,\mathrm{div}\,\mathbf r)_{\mathcal T_h}+\langle \widehat u_h,\mathbf r\cdot\mathbf n\rangle_{\partial\mathcal T_h} &=0 & & \forall \mathbf r \in \boldsymbol V_h,\\
\label{eq:1.8b}
-(\mathbf q_h,\nabla w)_{\mathcal T_h}+\langle \widehat{\mathbf q}_h\cdot\mathbf n,w\rangle_{\partial T_h} &=(f,w)_{\mathcal T_h}  &
\qquad &\forall w \in W_h,\\
\label{eq:1.8c}
-\langle\widehat{\mathbf q}_h\cdot\mathbf n,\widehat v\rangle_{\partial \mathcal T_h\setminus\Gamma} & = 0 & &\forall \widehat v \in M_h,\\
\label{eq:1.8d}
\langle \widehat u_h,\widehat v\rangle_\Gamma -\langle\smallfrac12\varphi_h+\mathcal K\varphi_h,\widehat v\rangle_\Gamma &=\langle\beta_0,\widehat v\rangle_\Gamma & & \forall \widehat v\in M_h,\\
\label{eq:1.8e}
-\langle\widehat{\mathbf q}_h\cdot\mathbf n,\phi\rangle_\Gamma+\omega(\varphi_h,\phi) &=\langle\beta_1,\phi\rangle_\Gamma & & \forall \phi \in Y_h.
\end{alignat}
Here $\widehat{\mathbf q}_h$ is defined on the boundaries of the elements using the expression
\begin{equation}\label{eq:1.8f}
\widehat{\mathbf q}_h:=\mathbf q_h|_K+\tau_{\partial K} (u_h|_K-\widehat u_h)\mathbf n_{\partial K}\qquad\mbox{on $\partial K$},
\end{equation}
\end{subequations}
where $\tau_{\partial K}:\partial K \to \mathbb R$ is a {\em non-negative} stabilization function that is constant on each edge/face of $\partial K$ and such that it is strictly positive on at least one edge/face of each triangle. The stabilization function $\tau$ can be double-valued on interlement edges/faces and $\widehat{\mathbf q}_h$ is in principle multiple valued, but  its normal component is made to be single valued through equation \eqref{eq:1.8c}. Note also that equations \eqref{eq:1.8c}-\eqref{eq:1.8d} are tested with the same space, but because of the integration domain, they produce together as many equations as the dimension of $M_h$. Using the local solvers for the HDG method \cite{CoGoLa:2009}, the system \eqref{eq:1.8} can be reduced to a system with $(\widehat u_h,\varphi_h)\in M_h \times Y_h$ as the only unknowns.
A discrete counterpart of Lemma \ref{lemma:1.1} holds.

\begin{lemma}\label{lemma:1.2}
If $(\mathbf q_h,u_h,\widehat u_h,\varphi_h)$ solves \eqref{eq:1.8} and the compatibility condition \eqref{eq:1.2} holds, then $\int_\Gamma \varphi_h=0$.
\end{lemma}

\begin{proof}
Testing equation \eqref{eq:1.8b} with $w=1$ and equation \eqref{eq:1.8e} and using Proposition \ref{prop:A.6}, the result follows.
\end{proof}

The next sections deal with the analysis of the method: we first show the basic discrete coercivity arguments (Section \ref{sec:2}), proceed to proving energy estimates based on the HDG projection (Section \ref{sec:3}) and finally use duality arguments to analyze convergence (and superconvergence) of some of the fields (Section \ref{sec:4}).

\section{Solvability of the discrete system}\label{sec:2}

The aim of this section is to show that, under some conditions on $\kappa$, the system \eqref{eq:1.8} has a unique solution for $h$ small enough. A relevant hypothesis will be the following:
\begin{equation}\label{eq:2.1}
\kappa_{\max}:=\|\kappa\|_{L^\infty(\Omega)} < 4.
\end{equation}
This hypothesis is related to coercivity of the underlying mixed formulation. We will discuss this hypothesis (and compare it with similar bounds for other coupled BEM-FEM schemes) in Section \ref{sec:5}, where we will also make some experiments related to it. For some of the forthcoming arugments, we will use the piecewise constant function $\mathfrak h:\partial\mathcal T_h \to \mathbb R$ given by $\mathfrak h|_e:=h_e$. We will pay attention to the quantities
\[
\tau_{\max}:=\max_{K\in \mathcal T_h} \|\tau\|_{L^\infty(\partial K)}=:\|\tau\|_{L^\infty(\partial\mathcal T_h)} \quad \mbox{and}\quad \|\mathfrak h \tau\|_{L^\infty(\partial\mathcal T_h)},
\]
noting that because of the shape regularity of the grid, we can bound $h_K \|\tau\|_{L^\infty(\partial K)}\le C \|\mathfrak h\tau\|_{L^\infty(\partial\mathcal T_h)}$ for all $K$.
Three more polynomial spaces will appear in our arguments:
\begin{alignat*}{4}
\mathcal P_k^\bot(K) &:= \{ u\in \mathcal P_k(K)\,:\, (u,w)_K=0 \quad \forall w \in \mathcal P_{k-1}(K)\},\\
\boldsymbol{\mathcal P}_k^\bot(K) &:=(\mathcal P_k^\bot(K))^d=\{ \mathbf q\in \boldsymbol{\mathcal P}_k(K)\,:\, (\mathbf q,\mathbf r)_K=0 \quad \forall \mathbf r \in \boldsymbol{\mathcal P}_{k-1}(K)\},\\
\mathcal R_k(\partial K)&:= \{ q:\partial K \to \mathbb R\,:\, q|_e\in \mathcal P_k(e) \quad \forall e \in \mathcal E(K)\},
\end{alignat*}
where in the last space, we have denoted $\mathcal E(K):=\{ e\in \mathcal E_h\,:\, e\subset \partial K\}$.

\begin{lemma}\label{lemma:2.1}
\begin{itemize}
\item[{\rm (a)}] If $q\in \mathcal P_k^\bot(K)$ and $q=0$ on $e\in \mathcal E(K)$, then $q=0$.
\item[{\rm (b)}] The following decomposition is orthogonal in $L^2(\partial K)$:
\[
\mathcal R_k(\partial K)=\{\mathbf v|_{\partial K} \cdot\mathbf n \,:\, \mathbf v \in \boldsymbol{\mathcal P}_k^\bot(K) \} \oplus \{ q|_{\partial K}\,:\, q\in \mathcal P_k^\bot(K)\}.
\]
\end{itemize}
\end{lemma}

\begin{proof} Part (a) is straightforward. Part (b) is Lemma 4.1 of \cite{CoSa:SB}.
\end{proof}

\begin{lemma}\label{lemma:2.2} 
Let $P_k:L^2(\Omega) \to W_h$ be the orthogonal projection onto $W_h$. Then
\[
\| u-P_k u\|_{\partial K}\le C h_K^{1/2}\|\nabla u\|_K \qquad \forall K \in \mathcal T_h, \forall u\in H^1(\Omega).
\]
\end{lemma}

\begin{proof}
Using the Bramble-Hilbert lemma (or a generalized Poincar\'e inequality), the trace theorem and an argument about finite dimensions, it is easy to prove that
\[
\| \widehat u-\widehat P_k \widehat u\|_{\partial\widehat K} \le C \| \nabla \widehat u\|_{\widehat K} \qquad \forall \widehat u\in H^1(\widehat K),
\]
where $\widehat K$ is the reference element and $\widehat P_k:L^2(\widehat K)\to \mathcal P_k(\widehat K)$ is the orthogonal projector. The result follows then by a scaling argument.
\end{proof}

\begin{lemma}[A discrete coercivity estimate]\label{lemma:2.3}
For discrete triples $(\mathbf q_h,u_h,\widehat u_h)\in \boldsymbol V_h\times W_h \times M_h$ satisfying
\begin{subequations}\label{eq:2.2}
\begin{alignat}{4}\label{eq:2.2a}
(\mathrm{div}\,\mathbf q_h,w)_{\mathcal T_h}+\langle\tau(u_h-\widehat u_h),w\rangle_{\partial \mathcal T_h} &=0 & & \forall w\in W_h,\\
\label{eq:2.2b}
\langle \mathbf q_h\cdot\mathbf n+\tau(u_h-\widehat u_h),\widehat v\rangle_{\partial\mathcal T_h\setminus\Gamma}&=0& \qquad & \forall \widehat v \in M_h,
\end{alignat}
\end{subequations}
and $u_\star\in H^1(\Omega)$, we consider the quadratic form:
\[
Q:=(\kappa^{-1}\mathbf q_h,\mathbf q_h)_{\mathcal T_h}+\langle\tau(u_h-\widehat u_h),u_h-\widehat u_h\rangle_{\partial\mathcal T_h}+\|\nabla u_\star\|_\Omega^2-\langle u_\star,\mathbf q_h\cdot\mathbf n+\tau(u_h-\widehat u_h)\rangle_\Gamma.
\]
If \eqref{eq:2.1} holds, then there exist $C>0$ and $D>0$ such that if $\|\mathfrak h\tau\|_{L^\infty(\partial \mathcal T_h)}\le D$,
\[
Q \ge C \Big( \|\mathbf q_h\|_\Omega^2+\|\nabla u_\star\|_\Omega^2+\langle\tau(u_h-\widehat u_h),u_h-\widehat u_h\rangle_{\partial\mathcal T_h}\Big).
\]
\end{lemma}

\begin{proof}
Condition \eqref{eq:2.2b} is equivalent to the fact that the discrete normal fluxes $\mathbf q_h\cdot\mathbf n+\tau(u_h-\widehat u_h)$ are single valued on internal faces/edges. Therefore
\[
E_h:=\langle u_\star,\mathbf q_h\cdot\mathbf n+\tau(u_h-\widehat u_h)\rangle_\Gamma=\sum_{K\in \mathcal T_h} E_{\partial K}, \mbox{ where } E_{\partial K}:=\langle u_\star,\mathbf q_h\cdot\mathbf n+\tau(u_h-\widehat u_h)\rangle_{\partial K}.
\]
If $P_k:L^2(\Omega)\to W_h$ is the orthogonal projection onto $W_h$ (as in Lemma \ref{lemma:2.2}), then
\begin{eqnarray*}
E_{\partial K} &=& \langle \mathbf q_h\cdot\mathbf n, u_\star\rangle_{\partial K}+\langle \tau(u_h-\widehat u_h),P_ku_\star\rangle_{\partial K}+\langle \tau(u_h-\widehat u_h),u_\star-P_ku_\star\rangle_{\partial K}\\
&=& \langle \mathbf q_h\cdot\mathbf n, u_\star\rangle_{\partial K}-\langle \mathrm{div}\,\mathbf q_h,P_k u_\star)_K+\langle \tau(u_h-\widehat u_h),u_\star-P_ku_\star\rangle_{\partial K}\\
&=& (\mathbf q_h,\nabla u_\star)_K+\langle \tau(u_h-\widehat u_h),u_\star-P_ku_\star\rangle_{\partial K},
\end{eqnarray*}
where we have used \eqref{eq:2.2a} and the fact that $\mathrm{div}\,\mathbf q_h \in \mathcal P_{k-1}(K)$. Therefore, by Lemma \ref{lemma:2.2} and Young's inequality
\begin{eqnarray*}
|E_{\partial K}|  & \le & \|\mathbf q_h\|_K \|\nabla u_\star\|_K + C (h_K\|\tau\|_{L^\infty(\partial K)})^{1/2} \langle \tau(u_h-\widehat u_h),u_h-\widehat u_h\rangle_{\partial K}^{1/2}\|\nabla u_\star\|_K\\
& \le & \delta^{-1} \|\mathbf q_h\|_K^2+ C^2\|\mathfrak h\tau\|_{L^\infty(\partial\mathcal T_h)} \underline\delta^{-1}\langle\tau(u_h-\widehat u_h),u_h-\widehat u_h\rangle_{\partial K} + \smallfrac14(\delta+\underline\delta)\|\nabla u_\star\|^2_K,
\end{eqnarray*}
for arbitrary $\delta, \underline\delta>0$. Therefore
\begin{eqnarray*}
Q &=& (\kappa^{-1}\mathbf q_h,\mathbf q_h)+\langle\tau(u_h-\widehat u_h),u_h-\widehat u_h\rangle_{\partial\mathcal T_h}+\|\nabla u_\star\|_\Omega^2-E_h \\
&\ge & (\kappa_{\max}^{-1}-\delta^{-1}) \|\mathbf q_h\|_\Omega^2+ (1-C^2\|\mathfrak h\tau\|_{L^\infty(\partial \mathcal T_h)}\underline\delta^{-1}) \langle\tau(u_h-\widehat u_h),u_h-\widehat u_h\rangle_{\partial\mathcal T_h}\\
& & +(1-\smallfrac14(\delta+\underline\delta))\|\nabla u_\star\|_\Omega^2\\
&\ge & \frac{4-\kappa_{\max}}{\kappa_{\max}(4+\kappa_{\max})} \|\mathbf q_h\|_\Omega^2+ \frac34\langle\tau(u_h-\widehat u_h),u_h-\widehat u_h\rangle_{\partial\mathcal T_h}+\frac{4-\kappa_{\max}}{16}\|\nabla u_\star\|_\Omega^2,
\end{eqnarray*}
where we have chosen
\[
\delta:=\frac{4+\kappa_{\max}}2 \qquad \underline\delta:=\frac{4-\kappa_{\max}}4, \qquad \|\mathfrak h\tau\|_{L^\infty(\partial \mathcal T_h)}\le \frac{4-\kappa_{\max}}{16 C^2}.
\]
This finishes the proof.
\end{proof}

\begin{proposition}
If \eqref{eq:2.1} holds, then there exists  $D>0$ such that for all $\|\mathfrak h\tau\|_{L^\infty(\partial \mathcal T_h)}\le D$, the system \eqref{eq:1.8} is uniquely solvable. In particular, if $\tau_{\max}\le C$, the system is uniquely solvable for $h$ small enough.
\end{proposition}

\begin{proof}
We only need to prove that the homogeneous system only admits the trivial solution. Let then $(\mathbf q_h,u_h,\widehat u_h,\varphi_h)$ be a solution of \eqref{eq:1.8} with zero right hand side. By Lemma \ref{lemma:1.2} it follows that $\omega(\varphi_h,\psi)=\langle\mathcal W\varphi_h,\psi\rangle_\Gamma$. We now test equations \eqref{eq:1.8} with $\mathbf q_h$, $u_h$, $\widehat u_h$, $ -\mathbf q_h\cdot\mathbf n-\tau(u_h-\widehat u_h)$ and $\varphi_h$ respectively, add the equations, and simplify the result to obtain
\[
(\kappa^{-1}\mathbf q_h,\mathbf q_h)_{\mathcal T_h}+\langle \tau(u_h-\widehat u_h),u_h-\widehat u_h\rangle_{\partial\mathcal T_h} + \langle -\smallfrac12 \varphi_h+\mathcal K \varphi_h,\mathbf q_h\cdot\mathbf n+\tau(u_h-\widehat u_h)\rangle_\Gamma+\langle \mathcal W\varphi_h,\varphi_h\rangle_\Gamma=0.
\]
Let now $u_\star:=\mathcal D\varphi_h$. Using then \eqref{eq:A.11b} and Proposition \ref{prop:A.6}, it follows that
\[
(\kappa^{-1}\mathbf q_h,\mathbf q_h)_{\mathcal T_h}+\langle \tau(u_h-\widehat u_h),u_h-\widehat u_h\rangle_{\partial\mathcal T_h}+\langle \gamma^-u_\star,\mathbf q_h\cdot\mathbf n+\tau(u_h-\widehat u_h)\rangle_\Gamma+\|\nabla u_\star\|_{\mathbb R^d\setminus\Gamma}^2=0,
\]
where $\gamma^-u_\star$ makes reference to the trace of $u_\star|_\Omega$ on $\Gamma$. We are now in the hypotheses of Lemma \ref{lemma:2.3}, which implies that $\mathbf q_h=\mathbf 0$,   $\nabla u_\star=\mathbf 0$, and $\tau(u_h-\widehat u_h)=0$ on $\partial K$ for all $K$ . By Proposition \ref{prop:A.1} and the fact that $\int_\Gamma \varphi_h=0$, this implies that $\varphi_h=0$.

Going back to \eqref{eq:1.8a}, it follows that
\begin{equation}\label{eq:2.3}
(\nabla u_h,\mathbf r)_K+\langle u_h-\widehat u_h,\mathbf r\cdot\mathbf n\rangle_{\partial K}=0\quad \forall \mathbf r\in \boldsymbol{\mathcal P}_k(K) \quad\forall K
\end{equation}
and therefore
\[
\langle u_h-\widehat u_h,\mathbf r\cdot\mathbf n\rangle_{\partial K}=0\quad \forall \mathbf r\in \boldsymbol{\mathcal P}_k^\bot(K) \quad\forall K.
\]
By Lemma \ref{lemma:2.1}(b), there exists $q \in \mathcal P_k^\bot(K)$ such that $u_h-\widehat u_h=q$ of $\partial K$. Since $\tau(u_h-\widehat u_h)=0$ and $\tau\neq 0$ on at least one $e\in \mathcal E(K)$, it follows by Lemma \ref{lemma:2.1}(a) that $q=0$ and therefore $u_h=\widehat u_h$ on $\partial\mathcal T_h$. Using this information in \eqref{eq:2.3} and testing with $\mathbf r=\nabla u_h$, it follows that $u_h$ is piecewise constant. At the same time, $u_h=\widehat u_h$ on $\partial\mathcal T_h$, which implies that both $u_h$ and $\widehat u_h$ are constant. Testing \eqref{eq:1.8d} with $\widehat v=1$ implies that this constant value has to vanish. This finishes the proof of uniqueness.
\end{proof} 

\section{Estimates by energy arguments}\label{sec:3}

Convergence analysis of the coupled HDG-BEM scheme follows the main lines of the projection-based analysis of HDG methods \cite{CoGoSa:2010} (see also \cite{CoGuSa:TA}). We start by recalling the local HDG projection \cite{CoGoSa:2010}. Given $(\mathbf q,u)$, we define $(\boldsymbol\Pi \mathbf q,\Pi u)\in \boldsymbol V_h\times W_h$ as the solution of the local problems
\begin{subequations}\label{eq:3.1}
\begin{alignat}{4}
\label{eq:3.1a}
(\boldsymbol\Pi\mathbf q,\mathbf r)_K &= (\mathbf q,\mathbf r)_K & \qquad & \forall \mathbf r \in \boldsymbol{\mathcal P}_{k-1}(K)\\
\label{eq:3.1b}
(\Pi u,w)_K &=(u,w) & & \forall w\in \mathcal P_{k-1}(K)\\
\label{eq:3.1c}
\langle \boldsymbol\Pi\mathbf q\cdot\mathbf n+\tau\Pi u,\widehat v\rangle_{\partial K} &=\langle \mathbf q\cdot \mathbf n+\tau u,\widehat v\rangle_{\partial K} & & \forall \widehat v \in \mathcal R_k(\partial K),
\end{alignat}
\end{subequations}
for all $K\in \mathcal T_h$. We also consider the orthogonal projections $P:\prod_{e\in \mathcal E_h} L^2(e) \to M_h$ and  $P_\Gamma :H^{1/2}(\Gamma) \to Y_h$. The error analysis is carried out by comparing the discrete solution with the projection of the exact solution, i.e., in terms of the quantities:
\[
\boldsymbol\varepsilon_h^q:=\boldsymbol\Pi\mathbf q-\mathbf q_h, \quad \varepsilon_h^u:=\Pi u-u_h, \quad \widehat{\varepsilon_h^u}:=P u-\widehat u_h, \qquad \varepsilon_h^\varphi:=P_\Gamma\varphi-\varphi_h.
\]
The error in the normal flux (recall \eqref{eq:1.8f}) is
\[
\widehat\varepsilon_h:=\boldsymbol\varepsilon_h^q\cdot\mathbf n+\tau(\varepsilon_h^u-\widehat{\varepsilon_h^u})=
\widehat\pi-\widehat{\mathbf q}_h\cdot\mathbf n, \quad\mbox{where}\quad \widehat\pi:=\boldsymbol\Pi \mathbf q\cdot \mathbf n+\tau(\Pi u-P u).
\]

\begin{lemma}\label{lemma:3.1}
The following inequalities hold for all $K$ and $e\in \mathcal E(K)$:
\begin{eqnarray}
h_e \|\widehat\pi-\mathbf q\cdot\mathbf n\|_e^2 & \le & C \Big( \|\mathbf q-\boldsymbol\Pi\mathbf q\|_K^2+h_K^2\|\nabla (\mathbf P_k\mathbf q-\mathbf q)\|_K^2\Big),\\
 \|\widehat\pi-\mathbf q\cdot\mathbf n\|_e^2 & \le & C h_K \|\nabla \mathbf q\|_K^2,
\end{eqnarray}
where $\mathbf P_k\mathbf q$ is the best $L^2(\Omega)^d$ approximation of $\mathbf q$ in $\boldsymbol V_h$.
\end{lemma}

\begin{proof}
Note also that by definition of the projections
\begin{equation}\label{eq:3.2}
\langle \widehat\pi ,\widehat v\rangle_{\partial K}=\langle \mathbf q \cdot\mathbf n,\widehat v\rangle_{\partial K} \qquad \forall \widehat v\in \mathcal R_k(\partial K) \quad \forall K.
\end{equation}
This shows that $\widehat\pi|_e$ is the  best $L^2(e)$ approximation of $\mathbf q\cdot \mathbf n_e$ on $\mathbb P_k(e)$. Using then a local trace inequality, we can easily bound for $e\in \mathcal E(K)$
\begin{equation}\label{eq:3.14}
h_e\| \widehat\pi-\mathbf q\cdot\mathbf n\|_e^2 \le h_e \| \mathbf P_k\mathbf q-\mathbf q\|_e^2\le C \Big( \|\mathbf P_k \mathbf q-\mathbf q\|_K^2 + h_K^2\|\nabla (\mathbf P_k\mathbf q-\mathbf q)\|_K^2\Big).
\end{equation}
The second inequality follows from a similar argument, comparing with $\mathbf P_0\mathbf q$ instead of $\mathbf P_k\mathbf q$.
\end{proof}

For simplicity, in some of the forthcoming arguments we will shorten $\widetilde{\mathcal K}:=\frac12\mathcal I+\mathcal K$. The following three discrete functionals will be relevant in the sequel as well:
\begin{subequations}\label{eq:3.2c}
\begin{alignat}{4}
c_1(\mathbf r)&:=(\kappa^{-1}(\boldsymbol\Pi\mathbf q-\mathbf q),\mathbf r)_{\mathcal T_h}, \\ 
c_2(\widehat v)&:=\langle\widetilde{\mathcal K}(P_\Gamma\varphi-\varphi),\widehat v\rangle_\Gamma,\\
c_3(\phi)&:=-\langle\widehat\pi-\mathbf q\cdot\mathbf n,\phi\rangle_\Gamma+\omega(P_\Gamma\varphi-\varphi,\phi).
\end{alignat}
\end{subequations}

\begin{proposition}[Energy inequality]\label{prop:3.1}
If \eqref{eq:2.1} holds, then there exists  $h_0$ such that for all $h\le h_0$,
\begin{equation}\label{eq:3.2b}
\|\boldsymbol\varepsilon_h^q\|_\Omega^2 + \langle\tau(\varepsilon_h^u-\widehat{\varepsilon_h^u}),\varepsilon_h^u-\widehat{\varepsilon_h^u}\rangle_{\partial\mathcal T_h}+\|\varepsilon_h^\varphi\|_{H^{1/2}(\Gamma)}^2
\le C \Big(c_1(\boldsymbol\varepsilon_h^q)+c_2(\widehat{\varepsilon_h})+c_3(\varepsilon_h^\varphi)\Big).
\end{equation}
\end{proposition}

\begin{proof} By the definition of the projections and \eqref{eq:3.2}, it follows that
\begin{subequations}\label{eq:3.3}
\begin{alignat}{4} 
(\kappa^{-1}\boldsymbol\Pi\mathbf q,\mathbf r)_{\mathcal T_h} -(\Pi u,\mathrm{div}\,\mathbf r)_{\mathcal T_h}+\langle P\,u,\mathbf r\cdot\mathbf n\rangle_{\partial\mathcal T_h} &=c_1(\mathbf r)& & \forall \mathbf r \in \boldsymbol V_h,\\
-(\boldsymbol\Pi\mathbf q,\nabla w)_{\mathcal T_h}+\langle \widehat\pi,w\rangle_{\partial \mathcal T_h} &=(f,w)_{\mathcal T_h}  &
\qquad &\forall w \in W_h,\\
-\langle\widehat\pi,\widehat v\rangle_{\partial \mathcal T_h\setminus\Gamma} & = 0 & &\forall \widehat v \in M_h,\\
\langle P u,\widehat v\rangle_\Gamma -\langle\widetilde{\mathcal K}P_\Gamma\varphi,\widehat v\rangle_\Gamma &=\langle\beta_0,\widehat v\rangle_\Gamma-c_2(\widehat v) & & \forall \widehat v\in M_h,\\
-\langle\widehat\pi,\phi\rangle_\Gamma+\omega(P_\Gamma\varphi,\phi) &=\langle\beta_1,\phi\rangle_\Gamma+c_3(\phi)& & \forall \phi \in Y_h.
\end{alignat}
\end{subequations}
Subtracting \eqref{eq:1.8} from equations \eqref{eq:3.3} we obtain the {\em error equations}:
\begin{subequations}\label{eq:3.5}
\begin{alignat}{4} 
\label{eq:3.5a}
(\kappa^{-1}\boldsymbol\varepsilon_h^q,\mathbf r)_{\mathcal T_h} -(\varepsilon_h^u,\mathrm{div}\,\mathbf r)_{\mathcal T_h}+\langle \widehat{\varepsilon_h^u},\mathbf r\cdot\mathbf n\rangle_{\partial\mathcal T_h} &=c_1(\mathbf r)& & \forall \mathbf r \in \boldsymbol V_h,\\
\label{eq:3.5b}
(\mathrm{div}\,\boldsymbol\varepsilon_h^q, w)_{\mathcal T_h}+\langle \tau(\varepsilon_h^u-\widehat{\varepsilon_h^u}),w\rangle_{\partial \mathcal T_h} &=0  &
\qquad &\forall w \in W_h,\\
\label{eq:3.5c}
-\langle\widehat{\varepsilon_h},\widehat v\rangle_{\partial \mathcal T_h\setminus\Gamma} & = 0 & &\forall \widehat v \in M_h,\\
\label{eq:3.5d}
\langle \widehat{\varepsilon_h^u},\widehat v\rangle_\Gamma -\langle\widetilde{\mathcal K}\varepsilon_h^\varphi,\widehat v\rangle_\Gamma &=-c_2(\widehat v) & & \forall \widehat v\in M_h,\\
\label{eq:3.5e}
-\langle\widehat{\varepsilon_h},\phi\rangle_\Gamma+\omega(\varepsilon_h^\varphi,\phi) &=c_3(\phi)& & \forall \phi \in Y_h.
\end{alignat}
\end{subequations}
Testing these equations with $\boldsymbol\varepsilon_h^q$, $\varepsilon_h^u$, $\widehat{\varepsilon_h^u}$, $-\widehat{\varepsilon_h}$ and $\varepsilon_h^\varphi$ respectively, adding them, and simplifying, we obtain
\begin{eqnarray}\label{eq:3.6}
(\kappa^{-1}\boldsymbol\varepsilon_h^q,\boldsymbol\varepsilon_h^q)_{\mathcal T_h}+\langle\tau(\varepsilon_h^u-\widehat{\varepsilon_h^u}),\varepsilon_h^u-\widehat{\varepsilon_h^u}\rangle_{\partial\mathcal T_h}+\langle\mathcal K \varepsilon_h^\varphi-\smallfrac12\varepsilon_h^\varphi,\widehat{\varepsilon_h}\rangle_\Gamma+\omega(\varepsilon_h^\varphi,\varepsilon_h^\varphi) & & \\
& & \hspace{-3cm}=c_1(\boldsymbol\varepsilon_h^q)+c_2(\widehat{\varepsilon_h})+c_3(\varepsilon_h^\varphi). \nonumber
\end{eqnarray}
Let now $\varepsilon_h^\star:=\mathcal D\varepsilon_h^\varphi$ and note that Proposition \ref{prop:A.6} and \eqref{eq:A.11b} imply that
\[
\omega(\varepsilon_h^\varphi,\varepsilon_h^\varphi)=\|\nabla\varepsilon_h^\star\|_{\mathbb R^d\setminus\Gamma}^2+\Big| \int_\Gamma \varepsilon_h^\varphi\Big|^2\qquad \mbox{and} \qquad \mathcal K \varepsilon_h^\varphi-\smallfrac12\varepsilon_h^\varphi=-\gamma^-\varepsilon_h^\star.
\]
Applying now the coercivity estimate of Lemma \ref{lemma:2.3}, it follows from \eqref{eq:3.6} that
\[
\|\boldsymbol\varepsilon_h^q\|_\Omega^2+\langle\tau(\varepsilon_h^u-\widehat{\varepsilon_h^u}),\varepsilon_h^u-\widehat{\varepsilon_h^u}\rangle_{\partial\mathcal T_h}+\omega(\varepsilon_h^\varphi,\varepsilon_h^\varphi)  \le C\Big(c_1(\boldsymbol\varepsilon_h^q)+c_2(\widehat{\varepsilon_h})+c_3(\varepsilon_h^\varphi)\Big). 
\]
Since the bilinear form $\omega$ is $H^{1/2}(\Gamma)$-coercive (Proposition \ref{prop:A.6}), the result follows.
\end{proof}

At this moment we have to introduce a technical hypothesis, 
\begin{equation}\label{eq:3.8}
h_\Gamma  \le C  \min  \{ h_K : K \in \mathcal T_h, \, \overline K \cap \Gamma \neq \emptyset\}, \qquad \mbox{i.e.,} \qquad \|\mathfrak h^{-1}\|_{L^\infty(\Gamma)}\le C h_\Gamma^{-1},
\end{equation}
meaning that the meshsize of $\Gamma_h$ is controlled by the minimum local meshsize of $\mathcal T_h$ on $\Gamma$.
This hypothesis is satisfied, for instance, if $\mathcal T_h$ is quasiuniform in a neighbordhood of $\Gamma$ and $\Gamma_h$ is equal the restriction of $\mathcal T_h$ to $\Gamma$, or is a refinement of this inherited mesh. We will also assume that
\begin{equation}\label{eq:3.50}
h_\Gamma \|\tau\|_{L^\infty(\Gamma)}\le C.
\end{equation}

\begin{theorem}\label{the:3.2}
Assume that \eqref{eq:2.1}, \eqref{eq:3.8} and \eqref{eq:3.50} hold. Then for $\|\mathfrak h \tau\|_{L^\infty(\partial\mathcal T_h)}$ small enough
\[
\|\boldsymbol\varepsilon_h^q\|_\Omega+\langle\tau(\varepsilon_h^u-\widehat{\varepsilon_h^u}),\varepsilon_h^u-\widehat{\varepsilon_h^u}\rangle_{\partial\mathcal T_h}^{1/2}+\|\varepsilon_h^\varphi\|_{H^{1/2}(\Gamma)}  \le C \Big(\mathrm{App}_h^q+
 \|P_\Gamma\varphi-\varphi\|_{H^{1/2}(\Gamma)}\Big),
\]
where
\[
\mathrm{App}_h^q := \|\boldsymbol\Pi\mathbf q-\mathbf q\|_\Omega+\Big(\sum_K h_K^2 \| \nabla( \mathbf P_k\mathbf q-\mathbf q)\|_K^2\Big)^{1/2}.
\]
\end{theorem}

\begin{proof}
We will derive the proof from Proposition \ref{prop:3.1} and some estimates related to the functionals defined in \eqref{eq:3.2c}. 
First of all
\begin{equation}\label{eq:3.9}
|c_1(\boldsymbol\varepsilon_h^q)| \le C \|\boldsymbol\Pi \mathbf q-\mathbf q \|_\Omega\|\boldsymbol\varepsilon_h^q\|_\Omega.
\end{equation}
Using Proposition \ref{prop:A.7} and a scaling argument (note the use of \eqref{eq:3.8} in the scaling argument), it follows that
\begin{eqnarray}
|c_2(\widehat{\varepsilon_h})| &\le &  \|\widetilde{K}(P_\Gamma\varphi-\varphi)\|_\Gamma ( \|\boldsymbol\varepsilon_h^q\cdot\mathbf n\|_\Gamma ++\|\tau\|_{L^\infty(\Gamma)}^{1/2}\langle\tau(\varepsilon_h^u-\widehat{\varepsilon_h^u}),\varepsilon_h^u-\widehat{\varepsilon_h^u}\rangle_\Gamma^{1/2})\nonumber
\\
\label{eq:3.10}
&\le & C \|P_\Gamma\varphi-\varphi\|_\Gamma \Big( h_\Gamma^{-1/2} \|\boldsymbol\varepsilon_h^q\|_\Omega +\|\tau\|_{L^\infty(\Gamma)}^{1/2}\langle\tau(\varepsilon_h^u-\widehat{\varepsilon_h^u}),\varepsilon_h^u-\widehat{\varepsilon_h^u}\rangle_\Gamma^{1/2}\Big).
\end{eqnarray}
An Aubin-Nitsche duality argument and properties of best approximation by piecewise polynomial functions shows that
\begin{equation}\label{eq:3.11}
\| P_\Gamma\varphi-\varphi\|_\Gamma \le C h_\Gamma^{1/2} \|P_\Gamma\varphi-\varphi\|_{H^{1/2}(\Gamma)}.
\end{equation}
For the third functional, we use Proposition \ref{prop:A.6} and \eqref{eq:3.2}
\begin{equation}\label{eq:3.12}
| c_3(\varepsilon_h^\varphi)|  \le  |\langle \widehat\pi-\mathbf q\cdot\mathbf n,\varepsilon_h^\varphi-P\varepsilon_h^\varphi\rangle_\Gamma | + C\|P_\Gamma\varphi-\varphi\|_{H^{1/2}(\Gamma)}\|\varepsilon_h^\varphi\|_{H^{1/2}(\Gamma)},
\end{equation}
where it is to be understood that the local projection $P\varepsilon_h^\varphi$ is only applied (and needed) on $\Gamma$. A local scaling argument
(cf  \cite[Theorem 3.2]{GaSa:2006}) shows that 
\begin{equation}\label{eq:3.13}
\| \mathfrak h^{-1/2} (\varepsilon_h^\varphi - P\varepsilon_h^\varphi)\|_\Gamma \le C \|\varepsilon_h^\varphi\|_{H^{1/2}(\Gamma)}.
\end{equation}
This inequality and Lemma \ref{lemma:3.1} can then be used in \eqref{eq:3.12} to prove that
\begin{equation}\label{eq:3.31}
| c_3(\varepsilon_h^\varphi)| \le C \Big( \mathrm{App}_h^q+\|P_\Gamma\varphi-\varphi\|_{H^{1/2}(\Gamma)}\Big) \|\varepsilon_h^\varphi\|_{H^{1/2}(\Gamma)}
\end{equation}
The result is then a consequence of Proposition \ref{prop:3.1} and the bounds \eqref{eq:3.9}, \eqref{eq:3.10}, \eqref{eq:3.11}, and \eqref{eq:3.31}. 
\end{proof}

\begin{corollary}\label{cor:3.4}
In the hypotheses of Theorem \ref{the:3.2},
\[
\|\mathfrak h^{1/2} (\mathbf q\cdot\mathbf n-\widehat{\mathbf q}_h\cdot\mathbf n)\|_{\partial \mathcal T_h} \le C  (\mathrm{App}_h^q+\|P_\Gamma\varphi-\varphi\|_{H^{1/2}(\Gamma)}).
\]
\end{corollary}

\begin{proof}
Adding and subtracting $\widehat\pi$, we can bound
\begin{eqnarray*}
\|\mathfrak h^{1/2} (\mathbf q\cdot\mathbf n-\widehat{\mathbf q}_h\cdot\mathbf n)\|_{\partial \mathcal T_h} &\le & \|\mathfrak h^{1/2}\widehat{\varepsilon_h}\|_{\partial\mathcal T_h}+\|\mathfrak h^{1/2}(\widehat\pi-\mathbf q\cdot\mathbf n)\|_{\partial\mathcal T_h}\\
& & \hspace{-3cm}\le \|\mathfrak h^{1/2}\boldsymbol\varepsilon_h^q\cdot\mathbf n\|_{\partial\mathcal T_h}+\|\mathfrak h\tau\|_{L^\infty(\partial\mathcal T_h)}\langle\tau(\varepsilon_h^u-\widehat{\varepsilon_h^u},\varepsilon_h^u-\widehat{\varepsilon_h^u}\rangle_{\partial\mathcal T_h}+\|\mathfrak h^{1/2}(\widehat\pi-\mathbf q\cdot\mathbf n)\|_{\partial\mathcal T_h}.
\end{eqnarray*}
A scaling argument (using the fact that $\boldsymbol\varepsilon_h^q\in \boldsymbol V_h$), plus Lemma \ref{lemma:3.1} and Theorem \ref{the:3.2} finish the proof.
\end{proof}

Let $\Gamma_1,\ldots,\Gamma_L$ be the faces (with $d=3$) or edges (when $d=2$) of $\Gamma$. We consider the space $X^m(\Gamma):=\prod_{\ell=1}^L H^m(\Gamma)$, endowed with its product norm.

\begin{corollary}\label{cor:3.5}
Assume that the hypotheses of Theorem \ref{the:3.2} hold and that the exact solution of \eqref{eq:1.6} satisfies: $\mathbf q\in H^{k+1}(\Omega)^d$, $u\in H^{k+1}(\Omega)$, $\varphi\in X^{k+2}(\Gamma)$. Then 
\begin{eqnarray*}
& & \hspace{-1cm} \|\mathbf q-\mathbf q_h\|_\Omega+ \|\varphi-\varphi_h\|_{H^{1/2}(\Gamma)} \\
& & \le C h^{k+1} (|\mathbf q|_{H^{k+1}(\Omega)}+\delta_\tau |u|_{H^{k+1}(\Omega)}) + C h_\Gamma^{k+3/2}\|\varphi\|_{X^{k+2}(\Gamma)},
\end{eqnarray*}
where $\delta_\tau:=\max_K \|\tau\|_{L^\infty(\partial K\setminus e_K)}$ and $e_K\in \mathcal E(K)$ is such that $\tau_{e_K}=\|\tau\|_{L^\infty(\partial K)}$.
\end{corollary}

\begin{proof}
This bound is a direct consequence of Theorem \ref{the:3.2}, using well-known estimates for the best approximation by piecewise polynomials and \cite[Theorem 2.1]{CoGoSa:2010}.
\end{proof}

\section{Estimates by duality arguments}\label{sec:4}

Consider the problem
\begin{subequations}
\begin{alignat}{4}
\kappa^{-1}\mathbf d+\nabla\Theta = 0 & \qquad & & \mbox{in $\Omega$},\\
 \mathrm{div}\,\mathbf d=\varepsilon_h^u  & & & \mbox{in $\Omega$},\\
\label{eq:4.1b}
\Delta \omega =0  & & & \mbox{in $\Omega_+$},\\
\label{eq:4.1bb}
\omega = c_\infty \Phi(r)+\mathcal O (r^{-d+1}) & & & \mbox{at infinity},\\
\label{eq:4.1c}
\Theta=\omega & & & \mbox{on $\Gamma$ },\\
\label{eq:4.1cc}
\mathbf d\cdot\mathbf n=-\partial_{\mathbf n}\omega & & &  \mbox{on $\Gamma$}.
\end{alignat}
\end{subequations}
By Proposition \ref{prop:A.22} and \eqref{eq:4.1b}-\eqref{eq:4.1cc}, we can represent $\omega=\mathcal D \Theta+\mathcal S (\mathbf d\cdot\mathbf n)$ in $\Omega_+$. Therefore, using the definition of the operator $\mathcal W$ --see \eqref{eq:A.11}-- and Proposition \ref{prop:A.7}, we can write
\begin{equation}\label{eq:4.2}
-\langle \mathbf d\cdot\mathbf n,\smallfrac12\psi+\mathcal K \psi\rangle_\Gamma +\langle\mathcal W \psi,\Theta\rangle_\Gamma =0 \quad \forall \psi\in H^{1/2}(\Gamma).
\end{equation}
We will assume the following regularity estimate
\begin{equation}\label{eq:4.3}
\|\Theta\|_{H^2(\Omega)}+\|\mathbf d\|_{H^1(\Omega)} \le C_{\mathrm{reg}} \| \varepsilon_h^u\|_\Omega.
\end{equation}
Since this is a transmission problem, this hypothesis is mainly related to the regularity of the diffusion coefficient $\kappa$. If $\kappa$ is smooth and equal to one in a neighborhood of $\Gamma$, then \eqref{eq:4.3} holds (see \cite[Proposition 3.4]{CoGuSa:TA} for a similar argument).

\begin{proposition}\label{prop:4.1}
In the hypotheses of Theorem \ref{the:3.2} and assuming the regularity estimate \eqref{eq:4.3}, it holds
\begin{equation}\label{eq:4.100}
\|\varepsilon_h^u\|_\Omega \le C h^{\min\{k,1\}} \Big(\mathrm{App}_h^q+\|P_\Gamma\varphi-\varphi\|_{H^{1/2}(\Gamma)}\Big).
\end{equation}
\end{proposition}

\begin{proof}
The gist of the proof consists of testing the error equations \eqref{eq:3.5} with $\boldsymbol\Pi\mathbf d$, $-\Pi\Theta$, $-P\Theta$ , $-(\boldsymbol\Pi\mathbf d+\tau(\Pi\Theta-P\Theta))$, and $-P_\Gamma\Theta$ respectively, add them and manipulate the result.
Let us first test the first error equation \eqref{eq:3.5a} with $\mathbf r=\boldsymbol\Pi\mathbf d$:
\[
(\kappa^{-1}(\mathbf q-\mathbf q_h),\boldsymbol\Pi\mathbf d)_{\mathcal T_h}-(\varepsilon_h^u,\mathrm{div}\,\boldsymbol\Pi\mathbf d)_{\mathcal T_h}+\langle \widehat{\varepsilon_h^u},\boldsymbol\Pi\mathbf d\cdot\mathbf n\rangle_{\partial \mathcal T_h}=0.
\]
Using integration by parts twice, the first condition in the definition of the HDG projection \eqref{eq:3.1a} and the fact that $\mathrm{div}\,\mathbf d=\varepsilon_h^u$, we can write
\[
(\mathrm{div}\,\boldsymbol\Pi\mathbf d,\varepsilon_h^u)_{\mathcal T_h}=-(\mathbf d,\nabla \varepsilon_h^u)_{\mathcal T_h}+\langle \varepsilon_h^u,\boldsymbol\Pi\mathbf d\cdot\mathbf n\rangle_{\partial\mathcal T_h}=\|\varepsilon_h^u\|_\Omega^2+\langle\varepsilon_h^u,(\boldsymbol\Pi\mathbf d-\mathbf d)\cdot\mathbf n\rangle_{\partial\mathcal T_h}.
\]
Adding these two equations, it follows that
\begin{eqnarray}\nonumber
\|\varepsilon_h^u\|_\Omega^2 &=& (\kappa^{-1}(\mathbf q-\mathbf q_h),\boldsymbol\Pi\mathbf d)_{\mathcal T_h}-\langle\varepsilon_h^u-\widehat{\varepsilon_h^u},(\boldsymbol\Pi\mathbf d-\mathbf d)\cdot\mathbf n\rangle_{\partial\mathcal T_h}+\langle \widehat{\varepsilon_h^u},\mathbf d\cdot\mathbf n\rangle_{\partial\mathcal T_h}\\
&=& (\kappa^{-1}(\mathbf q-\mathbf q_h),\boldsymbol\Pi\mathbf d)_{\mathcal T_h}+
\langle\tau(\varepsilon_h^u-\widehat{\varepsilon_h^u}),\Pi\Theta-P\Theta\rangle_{\partial\mathcal T_h}+\langle \widehat{\varepsilon_h^u},\mathbf d\cdot\mathbf n\rangle_\Gamma,
\label{eq:4.4}
\end{eqnarray}
where in the last equation we have applied \eqref{eq:3.1c} and the fact that $\widehat{\varepsilon_h^u}$ and $\mathbf d$ are single valued on interelement faces. We next test \eqref{eq:3.5b} with $w=\Pi\Theta$ and apply \eqref{eq:3.1b} and the fact that $\nabla\Theta=-\kappa^{-1}\mathbf d$:
\begin{eqnarray*}
0 &=& (\mathrm{div}\,\boldsymbol\varepsilon_h^q,\Theta)_{\mathcal T_h}+\langle\tau(\varepsilon_h^u-\widehat{\varepsilon_h^u}),\Pi\Theta\rangle_{\partial\mathcal T_h}\\
&=& (\kappa^{-1}\boldsymbol\varepsilon_h^q,\mathbf d)_{\mathcal T_h}+\langle \boldsymbol\varepsilon_h^q\cdot\mathbf n,P\Theta\rangle_{\partial\mathcal T_h}+
\langle\tau(\varepsilon_h^u-\widehat{\varepsilon_h^u}),\Pi\Theta\rangle_{\partial\mathcal T_h}.
\end{eqnarray*}
The following step consists of subtracting this equation from \eqref{eq:4.4} to obtain
\begin{equation}\label{eq:4.5}
\|\varepsilon_h^u\|_\Omega^2 = \underbrace{(\kappa^{-1}(\mathbf q-\mathbf q_h),\boldsymbol\Pi\mathbf d)_{\mathcal T_h}- (\kappa^{-1}\boldsymbol\varepsilon_h^q,\mathbf d)_{\mathcal T_h}}_{C_h^{(1)}}+\underbrace{\langle \widehat{\varepsilon_h^u},\mathbf d\cdot\mathbf n\rangle_\Gamma-\langle \boldsymbol\varepsilon_h^q\cdot\mathbf n+\tau(\varepsilon_h^u-\widehat{\varepsilon_h^u}),P\Theta\rangle_{\partial\mathcal T_h}}_{C_h^{(2)}}.
\end{equation}
{\em Bound for $C_h^{(1)}$.} Since $\kappa^{-1}\mathbf d=-\nabla\Theta$,
\begin{eqnarray*}
C_h^{(1)}&=& (\kappa^{-1}(\mathbf q-\mathbf q_h),\boldsymbol\Pi\mathbf d-\mathbf d)_{\mathcal T_h}+ (\boldsymbol\Pi\mathbf q-\mathbf q,-\kappa^{-1}\mathbf d)_{\mathcal T_h}\\
&=& (\kappa^{-1}(\mathbf q-\mathbf q_h),\boldsymbol\Pi\mathbf d-\mathbf d)_{\mathcal T_h}+(\boldsymbol\Pi\mathbf q-\mathbf q,
\nabla \Theta-\mathbf P_{k-1}\nabla\Theta)_{\mathcal T_h},
\end{eqnarray*}
where $\mathbf P_{k-1}$ is the local orthogonal projection on the spaces $\boldsymbol{\mathcal P}_{k-1}(K)$ if $k\ge 1$ and $\mathbf P_{-1}=0$. Note that the inclusion of $\mathbf P_{k-1}\nabla\Theta$ is possible because of the definition of the HDG projection \eqref{eq:3.1}. Therefore
\begin{eqnarray}
\nonumber
|C_h^{(1)}|  &\le& C h \|\mathbf q-\mathbf q_h\|_\Omega ( |\mathbf d |_{H^1(\Omega)}+ |\Theta|_{H^2(\Omega)})+ C h^{\min\{k,1\}} \|\boldsymbol\Pi\mathbf q-\mathbf q\|_\Omega |\Theta|_{H^2(\Omega)}\\
\label{eq:4.6}
& \le & C h^{\min\{k,1\}} \|\varepsilon_h^u\|_\Omega\Big( \|\mathbf q-\mathbf q_h\|_\Omega+ \|\boldsymbol\Pi\mathbf q-\mathbf q\|_\Omega\Big),
\end{eqnarray}
by regularity \eqref{eq:4.3} and \cite[Theorem 2.2]{CoGoSa:2010}. 

{\em Bound for $C_h^{(2)}$.} Let now $\Theta_h$ be the Cl\'ement approximation \cite{Clement:1975} on a $\mathbb P_1$ conforming finite element space on a triangulation $\widetilde{\mathcal T_h}$ whose restriction to the boundary is $\Gamma_h$. We then use \eqref{eq:3.1c} to write
\begin{equation}\label{eq:4.7}
C_h^{(2)} = \langle \widehat{\varepsilon_h^u},\widehat\pi_d\rangle_\Gamma-\langle\mathbf q\cdot\mathbf n -\widehat{\mathbf q}_h\cdot\mathbf n,\Theta_h\rangle_\Gamma+\langle\mathbf q\cdot\mathbf n -\widehat{\mathbf q}_h\cdot\mathbf n,\Theta_h-P\Theta\rangle_\Gamma,
\end{equation}
where $\widehat\pi_d:=\boldsymbol\Pi\mathbf d\cdot\mathbf n+\tau(\Pi\Theta-P\Theta)$. Testing the error equation \eqref{eq:3.5d}  with $\widehat\pi_d$,  and  using\eqref{eq:4.2}, it follows that
\begin{eqnarray}
\nonumber
\langle \widehat{\varepsilon_h^u},\widehat\pi_d\rangle_\Gamma & = &\langle (\smallfrac12 \mathcal I+\mathcal K)(\varphi-\varphi_h),\widehat\pi_d\rangle_\Gamma \\
\nonumber
& = &\langle (\smallfrac12\mathcal I+\mathcal K)(\varphi-\varphi_h),\widehat\pi_d-\mathbf d\cdot\mathbf n\rangle_\Gamma+\langle\mathcal W(\varphi-\varphi_h),\Theta\rangle_\Gamma\\
\label{eq:4.80}
&=& \langle (\smallfrac12\mathcal I+\mathcal K)(\varphi-\varphi_h),\widehat\pi_d-\mathbf d\cdot\mathbf n\rangle_\Gamma+\langle\mathcal W(\varphi-\varphi_h),\Theta-\Theta_h\rangle_\Gamma\\
& & +\langle\mathbf q\cdot\mathbf n -\widehat{\mathbf q}_h\cdot\mathbf n,\Theta_h\rangle_\Gamma ,
\nonumber
\end{eqnarray}
where we have used that
\[
\langle \mathbf q\cdot\mathbf n-\widehat{\mathbf q}_h\cdot\mathbf n,\phi\rangle_\Gamma+\langle\mathcal W(\varphi-\varphi_h),\phi\rangle_\Gamma=\omega(\varphi-\varphi_h,\phi)\quad\forall\phi\in Y_h,
\]
which is just \eqref{eq:3.5e} after using Lemmas \ref{lemma:1.1} and \ref{lemma:1.2}.
Inserting \eqref{eq:4.80}  in \eqref{eq:4.7} and applying Propositions \ref{prop:A.6} and \ref{prop:A.7}  we can bound
\begin{eqnarray}
\nonumber
|C_h^{(2)}| & \le & C \|\varphi-\varphi_h\|_{H^{1/2}(\Gamma)} (\|\widehat\pi_d-\mathbf d\cdot\mathbf n\|_{H^{-1/2}(\Gamma)}+\|\Theta-\Theta_h\|_{H^{1/2}(\Gamma)})\\
\label{eq:4.10}
& & + \|\mathfrak h^{1/2}(\mathbf q\cdot\mathbf n-\widehat{\mathbf q}_h\cdot\mathbf n)\|_\Gamma ( \| \mathfrak h^{-1/2} (P\Theta-\Theta)\|_\Gamma+\| \mathfrak h^{-1/2} (\Theta-\Theta_h)\|_\Gamma).
\end{eqnarray}
We now just need to bound all the terms in the right hand side of \eqref{eq:4.10}. A duality argument (Aubin-Nitsche trick), Lemma \ref{lemma:3.1}, and the regularity assumption \eqref{eq:4.3} show that
\begin{eqnarray}\nonumber
\|\widehat\pi_d-\mathbf d\cdot\mathbf n\|_{H^{-1/2}(\Gamma)} & =& \| P(\mathbf d\cdot\mathbf n)-\mathbf d\cdot\mathbf n\|_{H^{-1/2}(\Gamma)} \le C h^{1/2}_\Gamma \|\widehat\pi_d-\mathbf d\cdot\mathbf n\|_\Gamma \\ 
\label{eq:4.11}
& \le &C h\|\nabla\mathbf d\|_\Omega \le C h\|\varepsilon_h^u\|_\Omega.
\end{eqnarray}
Well-known properties of the Cl\'ement interpolant (see in particular \cite[Proposition 5.2]{CoGuSa:TA}) and \eqref{eq:4.3} prove also that
\begin{equation}\label{eq:4.12}
\|\Theta-\Theta_h\|_{H^{1/2}(\Gamma)}+\|\mathfrak h^{-1/2}(\Theta-\Theta_h)\|_\Gamma \le C h |\Theta|_{H^2(\Omega)}\le C h\|\varepsilon_h^u\|_\Omega.
\end{equation}
Finally, we just overestimate
\begin{equation}\label{eq:4.13}
 \| \mathfrak h^{-1/2} (P\Theta-\Theta)\|_\Gamma \le  \| \mathfrak h^{-1/2} (P_k\Theta-\Theta)\|_\Gamma \le C h^{\min\{1,k\}} |\Theta|_{H^2(\Omega)}\le 
 C h^{\min\{1,k\}} \|\varepsilon_h^u\|_\Omega,
\end{equation}
after applying a discrete trace inquality and \eqref{eq:4.3}. Bringing the bounds \eqref{eq:4.11}, \eqref{eq:4.12} and \eqref{eq:4.13} to \eqref{eq:4.10} and using Corollary \ref{cor:3.4}, we obtain the bound
\begin{equation}\label{eq:4.14}
|C_h^{(2)}|\le C h^{\min\{1,k\}} \|\varepsilon_h^u\|_\Omega (\mathrm{App}_h^q+ \|\varphi-\varphi_h\|_{H^{1/2}(\Gamma)}+\|P_\Gamma\varphi-\varphi_h\|_{H^{1/2}(\Gamma)})
\end{equation}
The bound \eqref{eq:4.100} is now a direct consequence of \eqref{eq:4.5}, \eqref{eq:4.6}, \eqref{eq:4.14} and Theorem \ref{the:3.2}. 
\end{proof}

Note that in absence of any kind of additional regularity hypothesis, the estimate of Proposition \ref{prop:4.1} can be easily repeated without the additional $h^{\min\{k,1\}}$ that provides superconvergence to the method. Also, as a consequence of Proposition \ref{prop:4.1}, the local postprocessing technique of \cite[Section 5]{CoGoSa:2010} can be applied here providing a $\mathcal O(h^{k+2})$ approximation of $u$.

\begin{corollary}\label{cor:5.2}
In the hypotheses of Proposition \ref{prop:4.1}, for $k\ge 1$,
\begin{equation}\label{eq:4.101}
\|\widehat\varepsilon_h^u\|_h:=\Big( \sum_{K\in \mathcal T_h} h_K \| \widehat\varepsilon_h^u\|_{\partial K}^2 \Big)^{1/2} \le C h \Big(\mathrm{App}_h^q+\|P_\Gamma\varphi-\varphi\|_{H^{1/2}(\Gamma)}\Big).
\end{equation}
\end{corollary}

\begin{proof}
The argument to derive \eqref{eq:4.101} from \eqref{eq:4.100}  can be taken verbatim from the proof of \cite[Theorem 4.1]{CoGoSa:2010}.\end{proof}

\begin{proposition}\label{prop:5.3}
In the hypotheses of Theorem \ref{the:3.2}
\[
\|\varphi-\varphi_h\|_\Gamma \le C h^{1/2} \Big( \mathrm{App}_h^q+\|P_\Gamma\varphi-\varphi\|_{H^{1/2}(\Gamma)}\Big)+ \|\widehat\varepsilon_h^u\|_\Gamma.
\]
\end{proposition}

\begin{proof}
The error equation \eqref{eq:3.5d} is equivalent to writing $\langle \widetilde{\mathcal K}(\varphi-\varphi_h),\widehat v\rangle_\Gamma=\langle \widehat\varepsilon_h^u,\widehat v\rangle_\Gamma$ for all $\widehat v\in M_h$. Therefore, by Lemmas \ref{lemma:1.1} and \ref{lemma:1.2} and Proposition \ref{prop:A.7}, we can bound
\begin{eqnarray*}
\|\varphi-\varphi_h\|_\Gamma &\le & C  \sup_{0\neq \phi\in H^0(\Gamma)} \frac{\langle \widetilde{\mathcal K}(\varphi-\varphi_h),\phi\rangle_\Gamma}{\|\phi\|_\Gamma}\\
& \le & C  \sup_{0\neq \phi\in H^0(\Gamma)}\frac{\langle \widetilde{\mathcal K}(\varphi-\varphi_h),\phi-P\phi\rangle_\Gamma}{\|\phi\|_\Gamma}+ C  \sup_{0\neq \phi\in H^0(\Gamma)}
\frac{\langle \widehat\varepsilon_h^u,P\phi\rangle_\Gamma}{\|\phi\|_\Gamma}\\
&\le & C \|\varphi-\varphi_h\|_{H^{1/2}(\Gamma)}\sup_{0\neq \phi\in H^0(\Gamma)}\frac{\|\phi-P\phi\|_{H^{-1/2}(\Gamma)}}{\|\phi\|_\Gamma} +\|\widehat\varepsilon_h^u\|_\Gamma.
\end{eqnarray*}
An Aubin-Nitsche duality argument shows then that $\|\phi-P\phi\|_{H^{-1/2}(\Gamma)}\le h^{1/2}\|\phi\|_\Gamma$. The proof then follows by Theorem \ref{the:3.2}.
\end{proof}

\begin{corollary}
In the hypotheses of Proposition \ref{prop:4.1} and assuming that $h\le C h_\Gamma$, for $k\ge 1$,
\[
\|\varphi-\varphi_h\|_\Gamma \le C h^{1/2}  \Big( \mathrm{App}_h^q+\|P_\Gamma\varphi-\varphi\|_{H^{1/2}(\Gamma)}\Big).
\]
\end{corollary}

\begin{proof}
It is a direct consequence of Proposition \ref{prop:5.3}, using Corollary \ref{cor:5.2}, the estimate $\|\widehat\varepsilon_h^u\|_\Gamma \le \|\mathfrak h^{-1}\|_{L^\infty(\Gamma)}^{1/2}\|\widehat\varepsilon_h^u\|_h$, and \eqref{eq:3.8}.
\end{proof}

\section{Some additional considerations}\label{sec:F}

\paragraph{On the potential representation.} The main drawback of a boundary integral formulation based on a potential ansatz on a polyhedral boundary is the expected lack of regularity of the associated density. This makes that the hypotheses for regularity of the solution in Corollary \ref{cor:3.5} might not be realistic. Note, however, that the numerical experiments shown in Section \ref{sec:5}, for which the density is unknown, show that the computation of the interior unknowns and of the exterior solution are not affected for the foreseeable lack of regularity of the density in the corners of the domain. On the other hand, the analysis allows for using a very refined grid $\Gamma_h$ near the corners. From the point of view of implementation, the case when $\Gamma_h$ is a refinement of the grid $\mathcal T_h$ restricted to $\Gamma$, avoids many of the complications of dealing with general non-matching grids.

\paragraph{Direct Boundary Integral formulation.} A possible remedy for the above problem is the use of a direct formulation. For formulations of mixed type (and the one leading to the HDG method is one of such), this has been explained in great detail in \cite{MeSaSe:2011}. The coupled formulation \eqref{eq:1.6} has to be modified with the following arguments. First of all, we represent the exterior solution by
\[
u_+=\mathcal D \psi+\mathcal S(\mathbf q\cdot\mathbf n+\beta_1) \qquad \psi \in H^{1/2}_0(\Gamma),
\]
where $\mathcal S$ is the single layer potential \eqref{eq:A.101}.
Second, we consider the coupled problem
\begin{subequations}\label{eq:5.1}
\begin{alignat}{4}
\label{eq:5.1a}
\kappa^{-1}\mathbf q +\nabla u^\circ & = 0 & & \mbox{ in $\Omega$},\\
\label{eq:5.1b}
\mathrm{div}\,\mathbf q &=f & & \mbox{ in $\Omega$},\\
\label{eq:5.1c}
u^\circ -\psi &=\beta_0 & & \mbox{ on $\Gamma$},\\
\label{eq:5.1d}
-\langle\mathbf q\cdot\mathbf n,\smallfrac12\phi+\mathcal K\phi\rangle_\Gamma + \omega(\varphi,\phi)&=\langle\beta_1,\smallfrac12\phi+\mathcal K\phi\rangle_\Gamma &\qquad & \forall \phi\in H^{1/2}(\Gamma).
\end{alignat}
\end{subequations}
Finally, the interior field $u^\circ$ is corrected by adding a constant, $u=u^\circ+c$, where
\[
 c:=\smallfrac1{|\Gamma|} \left(\langle 1,\smallfrac12\psi+\mathcal K\psi\rangle_\Gamma + \langle \mathbf q\cdot\mathbf n+\beta_1,\eta\rangle_\Gamma \right) \qquad \eta:=\int_\Gamma \Phi(|\cdot-\mathbf y|)\mathrm d\Gamma(\mathbf y),
\]
and $\Phi$ is the fundamental solution for the Laplacian \eqref{eq:A.100}. This gives the solution to \eqref{eq:1.1}. Without the correction only $\mathbf q$ and $u_+$ are correctly determined. If desired, it is possible to write $\psi+c=\gamma u_+$ and use this as a way of obtaining an approximation of the exterior trace. Equations \eqref{eq:1.8} can be easily modified to handle this reformulation. From the point of view of implementation, this requires a simple rearrangement of the matrices in \eqref{eq:1.8}. The fact that the integral operators appear also in the right hand side, in a term of the form $\langle \beta_1,\smallfrac12\phi+\mathcal K\phi\rangle_\Gamma$, with $\phi\in Y_h,$
can be easily handled by preprojecting the data function $\beta_1$ in the space $M_h$ restricted to the boundary $\Gamma$. Note that the HDG-BEM discretization of \eqref{eq:5.1} leads to a system whose matrix is the transpose of the one in Section \ref{sec:1}. Therefore, the quadratic form is the same, and all the energy arguments can be applied. The part of the analysis related to duality arguments requires some additional work though.

\paragraph{Non-symmetric RT-BEM.} A recent article \cite{CoGuSa:TA} studies the symmetric coupling of HDG and BEM (using two integral equations) in parallel to the coupling of Raviart-Thomas mixed elements with BEM. The latter had appeared in the literature long ago \cite{Meddahi:1996, CaFu:2000}. Among other things, \cite{CoGuSa:TA} provided an analysis of supercovergence in $L^2$ for the approximation of the variable $u$, and discussed the algorithmic advantages of hybridizing the RT method when coupled with BEM. Non-symmetric coupling of mixed elements with BEM has been proposed and studied in \cite{MeSaSe:2011}. The analysis of this paper can be modified to include the methods proposed in \cite{MeSaSe:2011}, thus providing some improved estimates that the variational techniques in that paper did not show.

\paragraph{On the diffusion parameter $\kappa$.} The final point for discussion is Hypothesis \eqref{eq:2.1}. First of all, let us mention that this inequality might not be sharp. We explore this in Section \ref{sec:5}. However, there is evidence that some hypothesis like this is consubstantial to the mixed formulation that we are using for the HDG-BEM coupling. A similar result, with a lower bound for the minimum value of $\kappa$ --i.e., an upper bound for $\|\kappa^{-1}\|_{L^\infty}$-- had been noticed in the context of FEM-BEM formulations: see \cite{Steinbach:2011} and \cite[Theorem 5.1]{GaHsSa:2012}. It has been recently proved \cite[Lemma 3.2]{OfSt:SB} that there is actually a threshold for the minimum value of $\kappa$ under which the formulation used in non-symmetric FEM-BEM loses its well-posedness. Nevertheless, if the threshold is crossed in points at a certain distance of the coupling boundary, a compactness argument can be invoked to show that the formulation and its Galerkin discretizations (for sufficiently refined grids) are well posed. While global compactness arguments (where lower order terms are added and subtracted to the equation) do not seem to be applicable for the current format of HDG analysis, it is possible that some equivalent ideas could be used to prove that the methods of this paper can be used by placing the boundary $\Gamma$ sufficiently far from the places where $\kappa$ is much larger than the exterior (unit) diffusivity.

\section{Experiments}\label{sec:5}

\paragraph{Convergence and superconvergence.} For this example the domain is the square $(0,1)\times (0,1)$. The boundary mesh $\Gamma_h$ is taken to be the restriction of $\mathcal T_h$ to $\Gamma$. The coarsest grid contains only two triangles. Other grids are obtained by uniform refinement, reaching up to 2048 elements and 3136 edges, of which 128 are boundary edges. Note that for a polynomial degree $k$, the dimension of the system that is solved (with $\widehat u_h$ and $\varphi_h$ as unknowns) is $k+1$ times the number of edges plus the number of boundary edges. We take the diffusion parameter $\kappa(x,y)=1+x^2$ and $u(x,y)=\exp( x+y)$ as exact solution in $\Omega$. The exterior solution is taken to be
\[
u^+(\mathbf x):= -\smallfrac1{2\pi}\log\frac{|\mathbf x-\mathbf x_1|}{|\mathbf x-\mathbf x_2|} \qquad \mathbf x_1=(0.3, 0.4), \quad \mathbf x_2=(0.7, 0.6).
\]
The density $\varphi$ is not known. To check errors we approximate the exterior solution at the observation point $\mathbf x_{\mathrm{obs}}=(-0.1,0.1)$. Computation of the exterior solution is done using high order quadrature on the formula for the double layer potential. We tabulate and plot the following errors:
\[
e_h^{\mathbf q}:=\frac{\|\mathbf q-\mathbf q_h\|_\Omega}{\|\mathbf q\|_\Omega}, \qquad e_h^{\hat u}:=\frac{\|u-\widehat u_h\|_h}{\| u\|_h},\qquad  e_h^+:=\frac{|u_+(\mathbf x_{\mathrm{obs}})- u_{+,h}( \mathbf x_{\mathrm{obs}})|}{|u_+( \mathbf x_{\mathrm{obs}})|}
\]
\[
\varepsilon_h^u :=\frac{\|\Pi u-u_h\|_\Omega}{\|u\|_\Omega}, \qquad  \varepsilon_h^{\hat u}:=\frac{\| Pu -u_h\|_h}{\| u\|_h}.
\]
The results are reported in Tables \ref{table:1} to \ref{table:3} for polynomials degree $k=0$ to $k=2$ respectively. Estimated convergence errors are computed using consecutive grids.
 
\begin{table}[ht] \small
  \centering \begin{tabular}{ c c c c c c c c c c} 
  \hline\hline 
   $e_h^{\mathbf q}$& ecr &$e_h^{\widehat u} $& ecr & $\varepsilon_h^u$ & ecr & $e_h^+$& ecr& $\varepsilon_h^{\widehat u}$ &ecr\\
    \hline
3.9102(-1)	&	-	&	2.5573(-1)	&	-	&	3.1210(-2)	&	-	&	6.3929(0)	&	-	&	4.7747(-2)	&	-	\\
2.2878(-1)	&	0.77	&	1.1908(-1)	&	1.10	&	1.4877(-2)	&	1.07	&	5.1969(-1)	&	3.62	&	2.6686(-2)	&	0.84	\\
1.1867(-1)	&	0.95	&	5.5497(-2)	&	1.10	&	6.2955(-3)	&	1.24	&	3.5613(-1)	&	0.55	&	8.4902(-3)	&	1.65	\\
6.0426(-2)	&	0.97	&	2.6790(-2)	&	1.05	&	3.5409(-3)	&	0.83	&	1.7215(-1)	&	1.05	&	3.8430(-3)	&	1.14	\\
3.0435(-2)	&	0.99	&	1.3157(-2)	&	1.03	&	1.8998(-3)	&	0.90	&	8.6852(-2)	&	0.99	&	1.9418(-3)	&	0.98	\\
1.5265(-2)	&	1.00	&	6.5192(-3)	&	1.01	&	9.8424(-4)	&	0.95	&	4.3598(-2)	&	0.99	&	9.9074(-4)	&	0.97	\\
   [1ex] \hline \end{tabular} 
\caption{Experiments for the lowest order ($k=0$) method. All errors behave like $\mathcal O(h)$.}\label{table:1}
\end{table}

\begin{table}[ht] \small
  \centering \begin{tabular}{ c c c c c c c c c c} 
  \hline\hline 
   $e_h^{\mathbf q}$& ecr &$e_h^{\widehat u} $&ecr & $\varepsilon_h^u$ & ecr & $e_h^+$&ecr & $\varepsilon_h^{\widehat u}$ &ecr \\
    \hline
9.1892(-2)	&	-	&	3.2673(-2)	&	-	&	4.8524(-3)	&	-	&	9.3891(-2)	&	-	&	9.9900(-3)	&	-	\\
2.5806(-2)	&	1.83	&	7.6059(-3)	&	2.10	&	7.6764(-4)	&	2.66	&	2.0486(-2)	&	2.20	&	1.6756(-3)	&	2.58	\\
6.7443(-3)	&	1.94	&	1.7824(-3)	&	2.09	&	1.1487(-4)	&	2.74	&	1.2999(-3)	&	3.98	&	2.4053(-4)	&	2.80	\\
1.7223(-3)	&	1.97	&	4.2887(-4)	&	2.06	&	1.5980(-5)	&	2.85	&	8.3227(-5)	&	3.97	&	3.2658(-5)	&	2.88	\\
4.3549(-4)	&	1.98	&	1.0507(-4)	&	2.03	&	2.1137(-6)	&	2.92	&	1.1326(-5)	&	2.88	&	4.2930(-6)	&	2.93	\\
1.0952(-4)	&	1.99	&	2.6000(-5)	&	2.01	&	2.7176(-7)	&	2.96	&	1.3269(-6)	&	3.09	&	5.5141(-7)	&	2.96	\\
   [1ex] \hline \end{tabular} 
\caption{Experiment for the case $k=1$. Errors for $\mathbf q$ and $u$ on $\partial \mathcal T_h$ behave like $\mathcal O(h^2)$. Comparison of $u$ with respect to the projections (in elements and on their boundaries), as well as the exterior potential superconverge like $\mathcal O(h^3)$.}\label{table:2}
 \end{table}

 \begin{table}[ht] \small
  \centering \begin{tabular}{ c c c c c c c c c c} 
  \hline\hline 
   $e_h^{\mathbf q}$& ecr &$e_h^{\widehat u} $&ecr & $\varepsilon_h^u$ & ecr & $e_h^+$&ecr & $\varepsilon_h^{\widehat u}$ &ecr \\
    \hline
1.2540(-2)	&	-	&	2.7770(-3)	&	-	&	5.9374(-4)	&	-	&	7.5351(-2)	&	-	&	9.3898(-4)	&	-	\\
1.7437(-3)	&	2.85	&	3.2768(-4)	&	3.08	&	3.9253(-5)	&	3.92	&	1.6677(-2)	&	2.18	&	9.6755(-5)	&	3.28	\\
2.3006(-4)	&	2.92	&	3.8082(-5)	&	3.11	&	2.7636(-6)	&	3.83	&	5.8133(-5)	&	8.16	&	7.6652(-6)	&	3.66	\\
2.8039(-5)	&	3.04	&	4.5361(-6)	&	3.07	&	1.4710(-7)	&	4.23	&	1.4569(-6)	&	5.32	&	4.1376(-7)	&	4.21	\\
3.5088(-6)	&	3.00	&	5.5514(-7)	&	3.03	&	9.0340(-9)	&	4.03	&	2.6695(-9)	&	9.09	&	2.5795(-8)	&	4.00	\\
4.3945(-7)	&	3.00	&	6.8673(-8)	&	3.02	&	5.7394(-10)	&	3.98	&	9.5947(-10)	&	1.48	&	1.6331(-9)	&	3.98	\\
   [1ex] \hline \end{tabular} 
\caption{Experiment for the case $k=2$. Errors for $\mathbf q$ and $u$ on $\partial \mathcal T_h$ behave like $\mathcal O(h^3)$. Comparison of $u$ with respect to the projections (in elements and on their boundaries), as well as the exterior potential superconverge like $\mathcal O(h^4)$. The exterior field behaves somewhat erratically, which might be due to unaccounted errors in BEM quadrature.}\label{table:3}
    \end{table}

 \paragraph{Tests related to the diffusion parameter.} For this example, the domain is the rectangle $(-3/2, 3/2)\times (-1,1)$ and we use $k=0$. We start with a fixed triangulation (produced with MATLAB's PDE Toolbox) with 936 elements. We take two different diffusion parameters: a constant value $\kappa(\mathbf x)\equiv \kappa_{cons}$ and a piecewise constant function
\[
\kappa(\mathbf x):=\left\{ \begin{array}{ll} \kappa_{int} & \mbox{in } (-3/4,3/4)\times (-1/2,1/2),\\
1 &\mbox{otherwise}.\end{array}\right.
\]
We note that the values $\kappa_{cons}=0$ and $\kappa_{int}= 0$ make the problem degenerate. We also note that the jump in the discontinuous diffusion coefficient is not resolved by the triangulation, i.e., we do not choose a triangulation with edges on the jump of the coefficient. As a test that might allow us to understand the effect of the diffusion parameter on the coupled system, we compute the condition number for increasing values of $\kappa_{cons}$ and $\kappa_{int}$ and plot them together in Figure \ref{fig:1}. It is clear from the results for constant diffusion that the condition \eqref{eq:3.1} is too restrictive --see also the comments in Section \ref{sec:F} concerning the recent results on nonsymmetric BEM-FEM--, but that growth of this parameter increases the condition number of the system significantly. It is also clear that if the diffusion coefficient grows far from the boundary, the problem is much better conditioned and that, for this case, the growth of conditioning appears to be linear in this parameter, which agrees with the basic fact that the matrix is an affine function of $\kappa^{-1}$ that seems not to degenerate as $\kappa$ grows.

We then take several concrete values (three constant diffusion parameters and two piecewise constant, with the same notation as above), and plot the condition number for uniformly refined grids. Results are shown in Figure \ref{fig:2}.

 \begin{figure}[htb]
\centering
\includegraphics[width=0.6\textwidth]{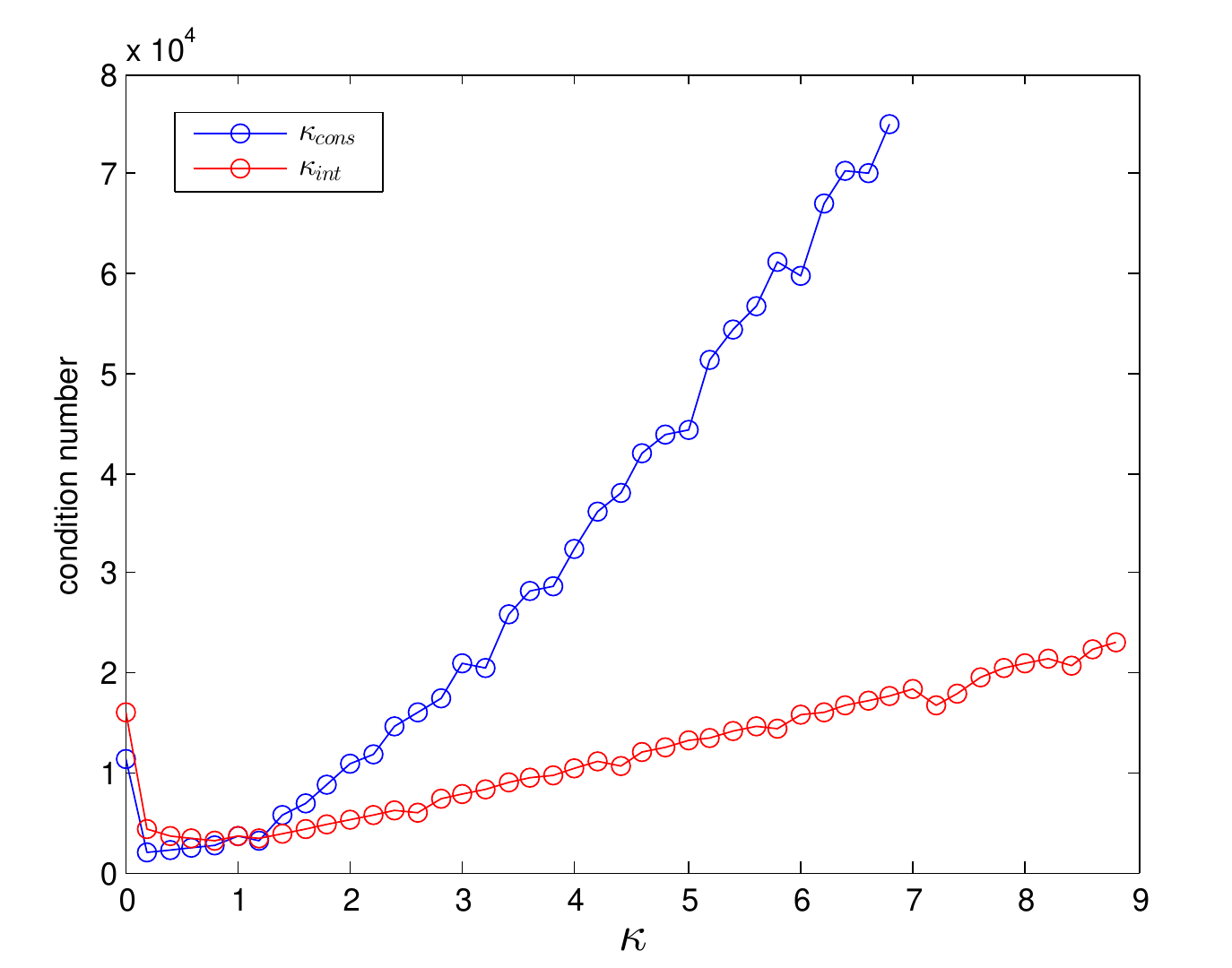} 
\caption{Condition number of the system matrix for growing values of the diffusion parameter in the entire domain ($\kappa_{cons}$) and in an interior subdomain ($\kappa_{int}$).} \label{fig:1}
\end{figure}

 \begin{figure}[htb]
\centering
\includegraphics[width=0.6\textwidth]{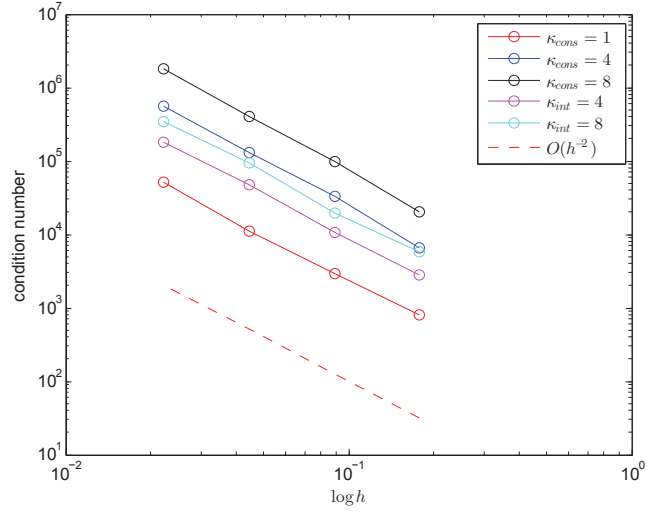} 
\caption{Condition number of the system matrix for different diffusion parameters and uniformly refined grids. All of them behave like $\mathcal O(h^{-2})$, as could be expected from a two dimensional elliptic problem.} \label{fig:2}
\end{figure}

\appendix

\section{Exterior and transmission problems: a compendium}\label{sec:A}

In this section we collect known results about layer potentials, as well as exterior and transmission problems associated to the Laplace equation. The results are straightforward consequences of results that are contained in \cite[Chapters 6 \& 8]{McLean:2000}.

\subsection{Layer potentials}

Let
\begin{equation}\label{eq:A.100}
\Phi(r)=\Phi_d(r):=\left\{ \begin{array}{ll} -1/(2\pi)\, \log r, & \mbox{when $d=2$},\\
1/(4\pi r), & \mbox{when $d=3$},\end{array}\right.
\end{equation}
be the fundamental solution of the Laplace equation. The double layer potential with density $\varphi\in H^{1/2}(\Gamma)$ is:
\begin{eqnarray}\nonumber
(\mathcal D \varphi)(\mathbf x)&:=&\int_\Gamma \nabla_{\mathbf y}\Phi(|\mathbf x-\mathbf y|)\cdot \mathbf n(\mathbf y)\, \varphi(\mathbf y)\,\mathrm d\Gamma(\mathbf y)\\ &=&\frac1{2(d-1)\pi}\int_\Gamma \frac{(\mathbf x-\mathbf y)\cdot\mathbf n(\mathbf y)}{|\mathbf x-\mathbf y|^d}\,\varphi(\mathbf y)\,\mathrm d\Gamma(\mathbf y), \qquad\forall\mathbf x \in \mathbb R^d\setminus\Gamma.\label{eq:A.00}
\end{eqnarray}
The single layer potential is defined with a duality product: for $\lambda \in H^{-1/2}(\Gamma)$, we define
$(\mathcal S \lambda)(\mathbf x):=\langle \lambda,\Phi(|\mathbf x-\cdot|)\rangle_\Gamma,$ on any $\mathbf x \in \mathbb R^d\setminus\Gamma.$
When $\lambda\in L^2(\Gamma)$, this duality product can be written in integral form
\begin{equation}\label{eq:A.101}
(\mathcal S \lambda)(\mathbf x)=\int_\Gamma \Phi(|\mathbf x-\mathbf y|)\,\lambda(\mathbf y)\,\mathrm d\Gamma(\mathbf y).
\end{equation}

\begin{proposition}\label{prop:A.1}
Let $\varphi\in H^{1/2}(\Gamma)$ and $u:=\mathcal D \varphi$. Then:
\begin{subequations}\label{eq:A.1}
\begin{alignat}{2}
& u \in H^1(\Omega)  \mbox{ and }   u \in H^1(\Omega_+\cap B(\mathbf 0;R)) \quad\forall R,\\
& \Delta u = 0   \mbox{ in $\mathbb R^d\setminus\Gamma$},\\
& u\in \mathcal C^\infty(\mathbb R^d\setminus\Gamma) \mbox{ and }  u=\mathcal O(r^{-d+1}) \mbox{ as $r=|\mathbf x|\to \infty$},\\
& \gamma^+ u-\gamma^- u=\varphi \mbox{ and } \partial_{\mathbf n}^+ u-\partial_{\mathbf n}^- u = 0   \mbox{ on $\Gamma$}.\label{eq:A.1d}
\end{alignat}
\end{subequations}
Moreover,
\begin{equation}\label{eq:A.2}
\mathcal D 1 = -\chi_\Omega=\left\{ \begin{array}{ll} -1 & \mbox{in $\Omega$},\\ 0, & \mbox{in $\Omega_+$}.\end{array}\right.
\end{equation}
\end{proposition}

\begin{proof} See \cite[Theorem 6.11]{McLean:2000} and \cite[Chapter 8]{McLean:2000}.
\end{proof}

The fact that constant densities produce a vanishing exterior solution of the Laplace equation motivates the introduction of the space
\begin{equation}\label{eq:A.0}
H^{1/2}_0(\Gamma):=\{ \varphi\in H^{1/2}(\Gamma)\,:\, \int_\Gamma \varphi =0\}.
\end{equation}

\subsection{Exterior solutions of the Laplace equation}

We consider functions $u_+:\Omega_+\to \mathbb R$ such that
\begin{equation}\label{eq:A.3}
\Delta u_+=0 \mbox{ in $\Omega_+$} \qquad \mbox{and}\qquad u_+\in H^1(\Omega_+\cap B(\mathbf 0;R)) \quad\forall R.
\end{equation}
Since, by Weyl's lemma, locally integrable solutions of the Laplace equation are $\mathcal C^\infty$, it then makes sense to impose a strong radiation condition at infinity. 
The general asymptotic condition we will deal with has the form
\begin{equation}\label{eq:A.3b}
u=c_\infty\Phi(r)+\mathcal O(r^{1-d}) \mbox{ as $r\to \infty$, uniformly in all directions}.
\end{equation}
Note that this condition includes logarithmically growing solutions when $d=2$, unless $c_\infty=0$. The incoming flux on $\Gamma$ is defined as
\begin{equation}\label{eq:A.6}
c_{\mathrm{flux}}:=-\langle \partial_{\mathbf n} u_+,1\rangle_\Gamma.
\end{equation}

\begin{proposition}\label{prop:A.22}
Let $u$ satisfy \eqref{eq:A.3} and \eqref{eq:A.3b}. Then:
\begin{itemize}
\item[{\rm (a)}] 
 $u_+$ admits the representation formula
\begin{equation}\label{eq:A.7}
u_+=\mathcal D \gamma u_+-\mathcal S \partial_{\mathbf n} u_+=\mathcal D(\gamma u_++c)-\mathcal S \partial_{\mathbf n} u_+\qquad \forall c \in \mathbb P_0(\Gamma).
\end{equation}
\item[{\rm (b)}] $c_\infty=c_{\mathrm{flux}}$ and therefore, a necessary and sufficient condition for $u$ to be decaying at infinity in the two dimensional case is $c_{\mathrm{flux}}=0$.
\item[{\rm (c)}] For every $\mathbf x_0\in \Omega$, there exists a unique $\phi\in H^{1/2}_0(\Gamma)$ such that
\begin{equation}
u_+=c_{\mathrm{flux}} \Phi(|\cdot-\mathbf x_0|)+\mathcal D \phi.
\end{equation}
Therefore, $u_+$ can be represented as a double layer potential if and only if $u_+=\mathcal O(r^{-d+1})$ at infinity.
\end{itemize}
\end{proposition}

\begin{proof}
Part (a) is a well-known representation formula \cite[Theorem 7.15]{McLean:2000}. The inclusion of any additive constant in the input of $\mathcal D$ follows from \eqref{eq:A.2}. Part (b) is straightforward using the strong integral form of the potentials. Part (c) can be easily proved by considering $u_+-c_{\mathrm{flux}} \Phi(|\cdot-\mathbf x_0|)$ as the solution of an exterior Neumann problem and using a double layer potential representation \cite[Theorem 8.19--8.21]{McLean:2000}.
\end{proof}

\subsection{Integral operators}\label{sec:A.4}

Because of the transmission conditions satisfied by the double layer potential \eqref{eq:A.1d}, we can define the operators
\begin{equation}\label{eq:A.11}
\mathcal W \varphi:=-\partial_{\mathbf n}^\pm \mathcal D \varphi, \qquad \mathcal K\varphi:=\smallfrac12 (\gamma^+ \mathcal D \varphi+\gamma^-\mathcal D\varphi).
\end{equation}
Note that the conditions \eqref{eq:A.1d} and the definition of $\mathcal K$ imply that
\begin{equation}\label{eq:A.11b}
\gamma^\pm \mathcal D \varphi=\pm \smallfrac12\varphi+\mathcal K\varphi. 
\end{equation}
When $\Gamma$ is a polyhedral ($d=3$) or polygonal ($d=2$) boundary,
the integral expression of the operator $\mathcal K$
\begin{equation}\label{eq:A.12}
(\mathcal K\varphi)(\mathbf x)=\frac1{2(d-1)\pi}\int_\Gamma \frac{(\mathbf x-\mathbf y)\cdot\mathbf n(\mathbf y)}{|\mathbf x-\mathbf y|^d}\,\varphi(\mathbf y)\,\mathrm d\Gamma(\mathbf y)
\end{equation}
is valid on all points $\mathbf x \in\Gamma$ that do not lie on edges. In particular \eqref{eq:A.12} holds almost everywhere on $\Gamma$. Also, if $\Gamma$ is a polyhedral boundary ($d=3$) and $\varphi,\phi\in L^\infty(\Gamma)$ are such that $\nabla_\Gamma \varphi,\nabla_\Gamma \phi \in L^\infty(\Gamma)^d$ (here $\nabla_\Gamma$ is the tangential gradient), we have an integral form for the bilinear form associated to $\mathcal W$:
\begin{equation}
\langle\mathcal W\varphi,\phi\rangle=\int_\Gamma \int_\Gamma (\mathbf n(\mathbf x)\times \nabla_\Gamma \varphi(\mathbf x))\cdot(\mathbf n(\mathbf y)\times \nabla_\Gamma\phi(\mathbf y))\,\Phi(|\mathbf x-\mathbf y|)\,\mathrm d\Gamma(\mathbf x)\mathrm d\Gamma(\mathbf y).
\end{equation}
In the two dimensional case, the bilinear form is
\begin{equation}
\langle\mathcal W\varphi,\phi\rangle=\int_\Gamma \int_\Gamma \partial_\tau \varphi(\mathbf x)\partial_\tau \phi(\mathbf y)\,\Phi(|\mathbf x-\mathbf y|) \,\mathrm d\Gamma(\mathbf x)\mathrm d\Gamma(\mathbf y),
\end{equation}
where $\partial_\tau$ is the tangential derivative on $\Gamma$.

\begin{proposition}[Properties of $\mathcal W$]\label{prop:A.6}
The operator $\mathcal W$ is bounded  $H^{1/2}(\Gamma) \to H^{-1/2}(\Gamma)$. Its kernel is the set of constant functions.
The bilinear form $\omega: H^{1/2}(\Gamma)\times H^{1/2}(\Gamma) \to \mathbb R$,
\[
\omega(\varphi,\phi):=\langle \mathcal W\varphi,\phi\rangle_\Gamma + \int_\Gamma \varphi\, \int_\Gamma \phi
\]
is bounded, symmetric, and coercive. Also  $\omega(\varphi,\varphi)=\langle\mathcal W \varphi,\varphi\rangle_\Gamma$ for all $\varphi \in H^{1/2}_0(\Gamma)$ and 
\[
\langle\mathcal W \varphi,\varphi\rangle_\Gamma =(\nabla u_\star,\nabla u_\star)_{\mathbb R^d\setminus\Gamma}, \quad \mbox{where } u_\star=\mathcal D \varphi, \quad \varphi\in H^{1/2}(\Gamma).
\]
\end{proposition}

\begin{proof} See \cite[Theorems 8.20 \& 8.21]{McLean:2000}.
\end{proof}

\begin{proposition}[Properties of $\mathcal K$] \label{prop:A.7}
The operator $\mathcal K$ is bounded  $H^{1/2}(\Gamma) \to H^{1/2}(\Gamma)$ and $H^0(\Gamma)\to H^0(\Gamma)$. Moreover,
\[
\langle \partial_{\mathbf n}^+ \mathcal S \lambda,\varphi\rangle_\Gamma = \langle \lambda,-\smallfrac12\varphi+\mathcal K\varphi\rangle_\Gamma \qquad \forall \lambda\in H^{-1/2}(\Gamma), \quad \varphi\in H^{1/2}(\Gamma).
\]
Finally
\begin{equation}\label{eq:A.70}
\|\xi-\smallfrac1{|\Gamma|}\int_\Gamma \xi\|_\Gamma \le C \sup_{0\neq \phi\in L^2(\Gamma)} \frac{\langle \frac12\xi+\mathcal K\xi,\phi\rangle_\Gamma}{\|\phi\|_\Gamma} \qquad \forall \xi \in H^0(\Gamma).
\end{equation}
\end{proposition}

\begin{proof}
Boundedness in $H^{1/2}(\Gamma)$ and the transposition property follow from the variational theory of layer potentials: see \cite[Theorems 6.11 \& 6.17]{McLean:2000}. For the $H^0(\Gamma)$ boundedness techniques of harmonic analysis are needed \cite{Verchota:1984}. The bound \eqref{eq:A.70} follows from the fact that $\smallfrac12\mathcal I+\mathcal K$ is Fredholm of index zero and its kernel is the set of constant functions.
\end{proof}

\subsection{Transmission problems}\label{sec:A.3}

Let $0\le \kappa \in L^\infty(\Omega)$ be such that $\kappa^{-1}\in L^\infty(\Omega)$. The data of the transmission problem are $f\in L^2(\Omega)$, $\beta_0\in H^{1/2}(\Gamma)$, $\beta_1\in H^{-1/2}(\Gamma)$.
We look for $u:\Omega\to\mathbb R$ and $u_+: \Omega_+\to\mathbb R$ satisfying: the exterior Laplace equation \eqref{eq:A.3} with radiation condition \eqref{eq:A.3b}, the interior elliptic equation
\begin{equation}\label{eq:A.9}
u\in H^1(\Omega) \qquad -\nabla\cdot(\kappa \nabla u)=f \mbox{ in $\Omega$},
\end{equation}
and the transmission conditions
\begin{equation}\label{eq:A.10}
\gamma u=\gamma u_++\beta_0\mbox{ on $\Gamma$} \qquad  \mbox{ and } \qquad \kappa \nabla u \cdot\mathbf n = \partial_{\mathbf n} u_++\beta_1 \mbox{ on $\Gamma$}.
\end{equation}
Because of \eqref{eq:A.9}-\eqref{eq:A.10}, it is easy to prove that $c_{\mathrm{flux}}=c_{\mathrm{data}}$, where
\begin{equation}
c_{\mathrm{data}}:=\int_\Omega f+\langle \beta_1,1\rangle_\Gamma.
\end{equation}

\begin{proposition}[Transmission problem]\label{prop:A.4}
The transmission problem looking for $u, u_+$ and $c_\infty$ satisfying \eqref{eq:A.3}, \eqref{eq:A.3b}, \eqref{eq:A.9}, and \eqref{eq:A.10}, is uniquely solvable.
\begin{itemize}
\item[{\rm (a)}] When $d=2$, the solution is decaying  ($c_\infty=0$ in \eqref{eq:A.3b}) if and only if $c_{\mathrm{data}}=0$.
\item[{\rm (b)}] When $d=3$, the solution satisfies $u_+=\mathcal O(r^{-2})$ as $r\to \infty$ if and only if $c_{\mathrm{data}}=0$. Finally, if $c_{\mathrm{data}}\neq 0$ and $\mathbf x_0\in \Omega$, we can write
\[
u_+=c_{\mathrm{data}} \Phi(|\cdot-\mathbf x_0|)+u_+^d,
\]
where $u_+^d$ satisfies \eqref{eq:A.3}, and $u_+^d=\mathcal O(r^{-2})$ at infinity. The transmission conditions can then be written
\[
\gamma u=\gamma u_+^d+\widetilde\beta_0\mbox{ on $\Gamma$} \qquad  \mbox{ and } \qquad \kappa \nabla u \cdot\mathbf n = \partial_{\mathbf n} u_+^d+\widetilde\beta_1 \mbox{ on $\Gamma$},
\]
where
\[
\widetilde\beta_0:=\beta_0+c_{\mathrm{data}} \Phi(|\cdot-\mathbf x_0|), \qquad \widetilde\beta_1:=\beta_1+c_{\mathrm{data}}\nabla \Phi(|\cdot-\mathbf x_0|)\cdot\mathbf n,
\]
and therefore
\[
\int_\Omega f +\langle \widetilde\beta_1,1\rangle_\Gamma=0.
\]
\end{itemize}
\end{proposition}

\begin{proof}
Existence and uniqueness of solution of \eqref{eq:A.3}, \eqref{eq:A.3b}, \eqref{eq:A.9}, and \eqref{eq:A.10} can be easily proved using a symmetric boundary-field formulation: see \cite[Section 1.5]{GaHs:1995} for the general methodology and the needed background results.
\end{proof}

\bibliographystyle{abbrv}
\bibliography{referencesHDGBEM}

\end{document}